\documentclass[11pt,reqno]{amsart}

\usepackage[left=1in, right=1in, top=1.3in, bottom=1.2in]{geometry}

\usepackage{graphicx}
\usepackage{indentfirst,csquotes}
\usepackage{amssymb,amsthm,amsmath}
\usepackage{xcolor,paralist,hyperref,fancyhdr,etoolbox}
\usepackage{lipsum}
\usepackage{dsfont}
\usepackage{tikz}
\usetikzlibrary{matrix}
\usetikzlibrary{positioning}
\usepackage{tikz-cd}
\usepackage{adjustbox}
\usepackage{thmtools}
\usepackage{thm-restate}
\usepackage{hyperref}
\usepackage{cleveref}
\usepackage[font=small]{caption} 

\newtheorem{theorem}{Theorem}[section]
\newtheorem{definition}[theorem]{Definition}
\newtheorem{lemma}[theorem]{Lemma}
\newtheorem{proposition}[theorem]{Proposition}

\newtheorem{remark}[theorem]{Remark}

\hypersetup{colorlinks=true, linkcolor=black, filecolor=black, urlcolor=black}

\begin{document}

\newcommand\bnode[3][2]{\node (#2) at (0, 0) {}; \fill (0, 0) circle (1pt) node[#1] {#3};}

\title{Large deviations for the largest singular value of sparse non-hermitian matrices}
\author[Hyungwon Han]{Hyungwon Han}
\address{Department of Mathematical Sciences, KAIST, South Korea}
\email{measure@kaist.ac.kr}

\begin{abstract}
We prove a large deviation principle for the largest singular value of sparse non-Hermitian random matrices, or directed Erd\H{o}s-R\'enyi networks in the constant average degree regime $p =\frac{d}{n}$ where $d$ is fixed. 
Entries are assumed to have Weibull distributions with tail decaying at rate $e^{-t^{\alpha}}$ for $\alpha >0$.
While the law of large number results agree with the largest eigenvalue of sparse Hermitian matrices given in \cite{ghn22, gn22}, large deviation results are surprisingly simpler, exhibiting a single transition at $\alpha =2$.
The rate function for undirected networks with $0<\alpha \leq 2$ involved a transition at $\alpha =1$ and a complicated variational formula due to the emergence of cliques with large edge-weights \cite{ghn22}. For directed networks, we introduce a clique reduction technique which reformulates the problem for undirected networks with maximum clique size $2$, and the rate function is greatly simplified.
For $\alpha >2$, both the law of large numbers and large deviation results are identical to the sparse Hermitian case.
Our results easily generalize to rectangular i.i.d. ensembles.
\end{abstract} 

\maketitle

\tableofcontents

\let\thefootnote\relax

\bigskip

\section{Introduction}

\subsection{Main results} \label{main results}

Random graphs and networks are of great importance in various areas of science and mathematics. Especially, the Erd\H{o}s-R\'enyi graph $\mathcal{G}(n,p)$ where each edge among $n$ vertices is selected with independent probability $p$ has been thoroughly analyzed over the past years. The spectral observables of adjacency matrices of weighted $\mathcal{G}(n,p)$ or sparse Hermitian matrices including the empirical measure and extremal eigenvalues are known to exhibit universality in the supercritical regime $p\gg \frac{\log n}{n}$. In the constant average degree regime $p= \frac{d}{n}$ where $d$ is fixed, universality no longer holds and the typical value for the largest eigenvalue of sparse Wigner matrices depends on entry distributions as computed in \cite{bbg21, gn22, ghn22}. An interesting feature of these results is that large deviation properties are readily deduced from the structural analysis of $\mathcal{G}(n,p)$, which was unusual in the study of dense random matrices.

A natural counterpart to $\mathcal{G}(n,p)$ is the Erd\H{o}s-R\'enyi digraph $\mathcal{G}_d(n,p)$ where each directed edge is chosen with independent probability $p$. The adjacency matrices of weighted $\mathcal{G}_d(n,p)$ correspond to sparse i.i.d. or non-Hermitian matrices. The spectral analysis of sparse non-Hermitian random matrices is difficult compared to the Hermitian case due to spectral instability and the lack of structural results for sparse digraphs. Only recently some results were proven for extremal eigenvalues of unweighted $\mathcal{G}_d(n,p)$ in \cite{bbk20, he23}. 

This paper aims to extend the spectral large deviation results for sparse Hermitian matrices in \cite{ghn22} toward sparse non-Hermitian matrices with general entries in the regime $p= \frac{d}{n}$. We focus on the largest singular value of the adjacency matrix of Erd\H{o}s-R\'enyi digraphs with general edge weights. Given an $n\times n$ matrix $Z$, let $\sigma_1 (Z)$ denote the largest singular value of $Z$ which is equal to the operator norm
\begin{align}
\| Z\|_{\mathrm{op}} =\sup_{\| v\|_2 =1}\| Zv\|_2 .
\end{align}
Our results will be stated in terms of this operator or spectral norm which we denote by $\| \cdot \|$ for simplicity.
We first state the large deviation results for square matrices then provide an extension to rectangular matrices in Section \ref{rectangular section}.

\vspace{12pt}

Consider the Erd\H{o}s-R\'enyi digraph $\mathcal{G}_{d}(n,p)$ with vertex set $[n]:= \lbrace 1,\ldots ,n\rbrace$ and each directed edge existing with probability $p$ independent of each other. Assume that $p=\frac{d}{n}$ where $d>0$ is a fixed constant. Then the adjacency matrix $X$ of $\mathcal{G}_{d}(n,\frac{d}{n})$ has i.i.d. Bernoulli entries with mean $\frac{d}{n}$. Let $Y$ be an $n\times n$ i.i.d. matrix with $\mathbb{E}Y_{ij}=0$ and $\mathbb{E}|Y_{ij}|^2=1$ for $1\leq i,j\leq n$ and define $Z=X\odot Y$ by the entrywise product $Z_{ij}=X_{ij}Y_{ij}$. We assume entries $Y_{ij}$ to have the Weibull distribution:

\begin{definition}
A random variable $W$ is said to have the Weibull distribution with shape parameter $\alpha >0$ if there exist constants $C_1,C_2>0$ such that for all $t\geq 1$,
\begin{align}
C_1 e^{-t^{\alpha}} \leq \mathbb{P}(|W| \geq t) \leq C_2 e^{-t^{\alpha}} .
\end{align}
\end{definition}

Both the law of large numbers and large deviation results for the largest singular value exhibit a single transition at the critical value $\alpha =2$. We divide our discussion into light-tailed ($\alpha >2$) and heavy-tailed ($\alpha \leq 2$) cases. Note that Gaussian weights ($\alpha =2$), treated separately in \cite{gn22} for sparse Hermitian matrices, can be absorbed in the heavy-tailed case after minor modifications.

\subsubsection{Light-tailed case, $\alpha >2$}

Let
\begin{align} \label{lambda light}
\lambda_{\alpha}^{\mathrm{light}} := 2^{\frac{1}{\alpha}}\alpha^{-\frac{1}{2}}(\alpha -2)^{\frac{1}{2}-\frac{1}{\alpha}}\frac{(\log n)^{\frac{1}{2}}}{(\log \log n)^{\frac{1}{2}-\frac{1}{\alpha}}} .
\end{align}
This expression will be the typical value of $\| Z\|$ for the light-tailed case. The following two theorems are the upper and lower tail large deviation results, respectively:

\begin{restatable}{theorem}{ltut} \label{ltut}
For any $\delta >0$,
\begin{align*}
\lim_{n\to\infty} -\frac{\log \mathbb{P} \left( \| Z\| \geq (1+\delta )\lambda_{\alpha}^{\mathrm{light}}\right) }{\log n}= (1+\delta )^2 -1.
\end{align*}
\end{restatable}

\begin{restatable}{theorem}{ltlt} \label{ltlt}
For any $0<\delta <1$,
\begin{align*}
\lim_{n\to\infty} \frac{1}{\log n}\left( \log \log \frac{1}{\mathbb{P}(\| Z\|\leq (1-\delta )\lambda_{\alpha}^{\mathrm{light}})}\right) =1-(1-\delta )^2 .
\end{align*}
\end{restatable}

We obtain the following law of large numbers result as a corollary:

\begin{restatable}{corollary}{ltlln} \label{ltlln}
We have
\begin{align*}
\lim_{n\to\infty}\frac{(\log \log n)^{\frac{1}{2}-\frac{1}{\alpha}}}{(\log n)^{\frac{1}{2}}}\| Z\| =2^{\frac{1}{\alpha}}\alpha^{-\frac{1}{2}}(\alpha -2)^{\frac{1}{2}-\frac{1}{\alpha}}
\end{align*}
in probability.
\end{restatable}

\subsubsection{Heavy-tailed case, $0<\alpha \leq 2$}

Define
\begin{align}
\lambda_{\alpha}^{\mathrm{heavy}}:= (\log n)^{\frac{1}{\alpha}}
\end{align}
as a counterpart of \eqref{lambda light} for the heavy-tailed case. Now we state the upper and lower tail large deviations:

\begin{restatable}{theorem}{htut}
\label{htut}
For any $\delta >0$,
\begin{align*}
\lim_{n\to\infty}-\frac{\log \mathbb{P}\left( \| Z\| \geq (1+\delta )\lambda_{\alpha}^{\mathrm{heavy}}\right)}{\log n}=(1+\delta )^{\alpha}-1.
\end{align*}
\end{restatable}

\begin{restatable}{theorem}{htlt} \label{htlt}
For any $0<\delta <1$,
\begin{align*}
\lim_{n\to\infty}\frac{1}{\log n}\left( \log \log \frac{1}{\mathbb{P}\left( \| Z\| \leq (1-\delta )\lambda_{\alpha}^{\mathrm{heavy}}\right)}\right) =1-(1-\delta )^{\alpha}.
\end{align*}
\end{restatable}

We deduce the following law of large numbers as a corollary:

\begin{restatable}{corollary}{htlln} \label{htlln}
We have
\begin{align*}
\lim_{n\to\infty}\frac{\| Z\|}{(\log n)^{\frac{1}{\alpha}}}=1
\end{align*}
in probability.
\end{restatable}

\begin{remark} \upshape{
The law of large numbers (LLN) result for the largest eigenvalue of sparse Wigner matrices in \cite{ghn22} are identical to our Corollaries \ref{ltlln} and \ref{htlln}. For the light-tailed case, the large deviation principle (LDP) is also identical to Theorems \ref{ltut} and \ref{ltlt}. In the heavy-tailed case, the lower tail large deviation agrees with Theorem \ref{htlt}, but the upper tail agrees with Theorem \ref{htut} only for $0<\alpha \leq 1$. For $1<\alpha \leq 2$, the rate function for the largest eigenvalue $\lambda_1(Z)$ has a variational formulation
\begin{align} \label{variational}
\lim_{n\to\infty}-\frac{\log \mathbb{P}\left( \lambda_1 (Z)\geq (1+\delta )\lambda_{\alpha}^{\mathrm{heavy}}\right)}{\log n}=\min_{k=2,3,\cdots }\psi_{\alpha ,\delta }(k)
\end{align}
where
\begin{align} \label{chicken}
\psi_{\alpha ,\delta }(k) := \frac{k(k-3)}{2}+\frac{1}{2}(1+\delta )^{\alpha}\phi_{\beta /2}(k)^{1-\alpha}
\end{align}
in which $\beta$ is the H\"older-conjugate of $\alpha$ and 
\begin{align}
\phi_{\theta}(k) := \sup_{v\in \mathbb{R}^k, v=(v_1, \cdots ,v_k), \| v\|_1 =1}\sum_{i,j\in [k], i\neq j}|v_i|^{\theta}|v_j|^{\theta}.
\end{align}
Our rate function $(1+\delta )^{\alpha}-1$ corresponds to $\psi_{\alpha ,\delta}(2)$ which is explicitly calculated using $\phi_{\beta /2}(2)=2^{1-\beta}$.
}
\end{remark}

\begin{remark} \upshape{
The spectral radius $\rho (Z)$ of a matrix $Z$ is defined as
\begin{align*}
\rho (Z) =\max_{1\leq i\leq n}|\lambda_i (Z)|,
\end{align*}
where $\lambda_1 (Z), \ldots ,\lambda_n(Z)$ are the eigenvalues of $Z$ and satisfies 
$\rho (Z) \leq \sigma_1 (Z) = \| Z\| $.
The limiting behavior of $\rho$ for sparse non-Hermitian random matrices in the regime $p=\frac{d}{n}$ is unknown for general entry distributions. The only result available in the literature is the following upper tail bound for the spectral radius of $\mathcal{G}_d(n,\frac{d}{n})$ with \emph{bounded} edge-weights:
\begin{theorem}[{\cite[Theorem 2.11]{bbk20}}] \label{aihp20 theorem}
Let $H=(H_{ij})_{1\leq i,j\leq n}$ be an i.i.d. matrix with mean-zero random variables satisfying $\mathbb{E}|H_{ij}|^2 \leq \frac{1}{n}$ and $\max_{ij}|H_{ij}|\leq \frac{1}{q}$ a.s. Then for $0<q\leq n^{1/10}$ and $\varepsilon \geq 0$, there exist some universal positive constants $C,c$ so that
\begin{align*}
\mathbb{P}\left( \rho (H)\geq 1+\varepsilon \right) \leq Cn^{2-cq\log (1+\varepsilon )}.
\end{align*}
\end{theorem}
Using the tail estimate
\begin{align*}
\mathbb{P}\left(|W|>r(\log n)^{\frac{1}{\alpha}}\right) \asymp e^{-r^{\alpha}\log n} \ll n^{-1}
\end{align*}
for some $r>1$ where $W$ is a Weibull distribution with parameter $\alpha$, and the fact that our matrix $Z$ has typically $O(n)$ nonzero entries, we may truncate the entries to have absolute value smaller than $r(\log n)^{\frac{1}{\alpha}}$. Applying Theorem \ref{aihp20 theorem} to the normalized matrix $\frac{1}{r(\log n)^{\frac{1}{\alpha}}}Z$ yields
\begin{align} \label{spectral radii}
\mathbb{P}\left(\rho (Z)\geq (1+\varepsilon )r(\log n)^{\frac{1}{\alpha}}\right) \leq Cn^{2-cr\log (1+\varepsilon )}
\end{align}
for $\varepsilon \geq 0$ and some universal constants $C,c>0$. This shows that $\rho (Z) \lesssim (\log n)^{\frac{1}{\alpha}}$ with high probability for any $\alpha >0$. 

Since $\rho (Z)\leq \| Z\|$, Corollaries \ref{ltlln} and \ref{htlln} provide an upper tail bound for $\rho$. For heavy-tailed weights Theorem \ref{htlln} implies $\rho (Z)\leq (\log n)^{\frac{1}{\alpha}}$ where the order agrees with \eqref{spectral radii} and an explicit constant is provided. However, in the light-tailed case Corollary \ref{ltlln} shows that $\rho (Z) \ll \| Z\|$. While \cite[Corollary 3.5]{bbk20} provides an estimate for all eigenvalues of Erd\H{o}s-R\'enyi digraphs, even the order of the spectral radius is unknown in the presence of weights. }
\end{remark}

\subsection{Extension to rectangular matrices} \label{rectangular section}

The results of the previous section easily extends to rectangular matrices with i.i.d. entries. First we prove the following lemma regarding the spectral norm of minors of a matrix:

\begin{lemma} \label{submatrix}
Let $A=(a_{ij})_{1\leq i\leq m, 1\leq j\leq n}$ be an $m\times n$ matrix with real entries and $B$ be its top left $p\times q$ minor, that is, $B=(a_{ij})_{1\leq i\leq p, 1\leq j\leq q}$ where $1\leq p\leq m$ and $1\leq q\leq n$. Then $\| A\| \geq \| B\|$.
\end{lemma}

\begin{proof}
Let $x=(x_1,\ldots ,x_q)\in\mathbb{R}^q$ with $\| x\|_2=1$ satisfy $\| Bx\|_2 =\| B\|$. Then
\begin{align*}
\| A\|^2 & =\sup_{y\in \mathbb{R}^n}\| Ay\|_2^2 =\sup_{y\in \mathbb{R}^n}\sum_{i=1}^{m}\left( \sum_{j=1}^{n}a_{ij} y_j\right)^2 \geq \sup_{y\in \mathbb{R}^q\times \lbrace 0\rbrace^{n-q}}\sum_{i=1}^{m}\left( \sum_{j=1}^{q}a_{ij}y_j\right)^2 \\
& \geq \sup_{y\in\mathbb{R}^q}\sum_{i=1}^{p}\left( \sum_{j=1}^{q}a_{ij}y_j\right)^2 \geq \sum_{i=1}^{p}\left( \sum_{j=1}^{q}a_{ij}x_j\right)^2 =\| Bx\|_2^2 =\| B\|^2
\end{align*}
where we chose $y=(x_1,\ldots ,x_q,0,\ldots ,0)$ in the first inequality.
\end{proof}

Using this, we may prove the following extension to rectangular matrices.

\begin{theorem}
Let $A_{m\times n}$ be an $m\times n$ matrix with i.i.d. weights having Weibull distributions with shape parameter $\alpha >0$ and $\frac{m}{n}\to \gamma \in (0,\infty )$ as $m,n\to \infty$. Then the results of Section \ref{main results} hold for $A_{m\times n}$. 
\end{theorem}

\begin{proof}
Let $M_{m\wedge n}$ be the top left $(m\wedge n)\times (m\wedge n)$ minor of $A_{m\times n}$ and $M_{m\vee n}$ be an $(m\vee n)\times (m\vee n)$ square i.i.d. matrix containing $A_{m\times n}$. Then by the lemma above,
\begin{align*}
\| M_{m\wedge n}\| \leq \| A_{m\times n}\| \leq  \| M_{m\vee n}\| .
\end{align*}
Applying the results of Section \ref{main results} to $M_{m\wedge n}$, $M_{m\vee n}$ and using $m\wedge n, m\vee n=\Theta (n)$, the proof is complete. 
\end{proof}

\subsection{Related results}

\subsubsection{Hermitian matrices}

Large deviation results for random matrices emerged from the seminal work of Arous and Guionnet \cite{ag97} who showed the LDP for the empirical measure of symmetric random matrices with Gaussian entries. The LDP for the largest eigenvalue of the Gaussian Orthogonal Ensemble (GOE) was shown in \cite{adg01} and later extended to general Wigner matrices with sub-Gaussian entries in \cite{au16, gh20, agh21}.

Recently, there has been great progress in the large deviation theory for sparse random matrices. 
In the supercritical regime $p\gg \frac{\log n}{n}$, Augeri and Basak \cite{ab23} obtained the LDP for the largest eigenvalue of sparse Hermitian matrices. In the same regime, large deviations for the empirical spectral measure was established by Augeri \cite{au24} using the method of quadratic vector equations.

Approaches in the constant average degree regime $p= \frac{d}{n}$ are mainly based on structural properties of sparse random graphs. The LLN behavior of the largest eigenvalue of Erd\H{o}s-R\'enyi graphs was first established in \cite{ks03}. In \cite{bbg21}, the authors proved a large deviation result for the extremal eigenvalues of $\mathcal{G}(n,p)$ including the constant average degree regime. As an extension of these results, the LDPs for the largest eigenvalue of sparse Hermitian matrices were computed for Gaussian \cite{gn22} and Weibull \cite{ghn22} entry distributions. Besides Erd\H{o}s-R\'enyi graphs and networks, the LLN for the largest eigenvalue of weighted random $d$-regular graphs with $d$ fixed was proven in \cite{ln23}.

\subsubsection{Non-Hermitian matrices}

The LDP for the empirical measure of Ginibre ensembles (whose convergence is governed by the circular law) was shown in \cite{az98}.
Similar results for singular values of rectangular i.i.d. matrices, or eigenvalues of sample covariance matrices (i.e., Wishart matrices) have been studied, starting from the Marchenko-Pastur law for the empirical measure \cite{mp67}. The LLN behavior of the largest and smallest singular values were analyzed in \cite{bky88, by93}. Large deviations for the largest eigenvalue of Wishart matrices were computed for Gaussian entries in \cite{mv09} and sharp sub-Gaussian entries in \cite{gh20}.
Recently, Husson and McKenna generalized the results to sample covariance matrices of the form $Z^{\intercal}\Gamma Z$ \cite{hm23}.

Despite these advances, there are still only a few results on the spectral properties of sparse non-Hermitian matrices, or weighted Erd\H{o}s-R\'enyi digraphs. As in the case of sparse Hermitian matrices, universality of the empirical measure holds in the regime $p\gg \frac{1}{n}$. Basak and Rudelson \cite{br19} showed a sparse version of the circular law for $p\gg \frac{\log^2n}{n}$ which was extended to $p\gg \frac{1}{n}$ by Rudelson and Tikhomirov \cite{rt19}. For extremal spectral observables, He \cite{he23} showed edge universality and LLN results for the spectral radii of $\mathcal{G}_d(n,p)$ in the regime $n^{-1+\varepsilon }\leq p\leq 1/2$. G\"{o}tze and Tikhomirov \cite{gt23} provided estimates for the largest and smallest singular values of sparse rectangular matrices in the regime $p\gg \frac{\log n}{n}$. 

Up to the author's knowledge, the only known result in the regime $p =\frac{d}{n}$ is an upper tail bound for the spectral radii of Erd\H{o}s-R\'enyi digraphs and networks with bounded edge-weights in \cite{bbk20}. It seems that there is no spectral large deviation result for sparse non-Hermitian matrices prior to this paper.

\subsection{Notations}

The notations of this paper will generally follow those of \cite{ghn22}.
Given a graph or digraph $G$, let $V=V(G)$ and $E=E(G)$ be the set of vertices and edges, respectively. The vertex set of a graph or digraph with $n$ vertices is considered as $V=[n]:= \lbrace 1,2,\ldots ,n\rbrace$. A network with underlying graph $G=(V,E)$ and conductance matrix $A$ is written as $G=(V,E,A)$.

For a graph $G=(V,E)$ and a vertex $v\in V(G)$, let $d(v)$ denote the degree of $v$ and $d_1(G)$ be the maximum degree of $G$. An undirected edge between the vertices $i,j$ where $i<j$ will be written as $(i,j)$.
Given a digraph $G=(V,E)$ and $v\in V(G)$, define $d^- (v)$ and $d^+(v)$ as the indegree and outdegree of $v$, respectively. The degree of $v$ is given by $d(v)=d^-(v)\vee d^+(v)$ and the degree of $G$, $d_1(G)$, is defined as in the undirected case. A directed edge from $i$ to $j$ is denoted by $(i,j)$. A digraph will be called a star if the corresponding undirected graph ignoring direction is an undirected star.

\subsection{Organization}

In Section 2, we provide a brief overview of the main ideas of our proof. Section 3 is devoted to spectral and structural properties of Erd\H{o}s-R\'enyi graphs and digraphs including the symmetrization and clique reduction processes which are the main ingredients of this work. Section 4 and 5 deal with proofs for the light-tailed and heavy-tailed cases, respectively. Finally, Section 6 is an appendix on binomial and Weibull tail estimates used throughout this paper.

\subsection{Acknowledgements}

The author thanks Kyeongsik Nam for helpful suggestions and advice.

\section{Overview of proof}

\subsection{Clique reduction and symmetrization}

Our main objective is to construct undirected networks corresponding to each connected component of the underlying Erd\H{o}s-R\'enyi digraph and to apply the estimates from \cite{ghn22}. This is possible since the spectral radius of a Hermitian matrix is equal to its spectral norm, and $-\lambda_n$ satisfies the same results as $\lambda_1$ where $\lambda_1, \lambda_n$ denote the largest and smallest eigenvalues, respectively.

While a clique (i.e., complete subgraph) with large edge-weights leads to a variational formulation of the rate function \eqref{chicken} for Hermitian matrices, we use a clique reduction process which transforms directed networks into undirected networks with maximum clique size $2$, thus reducing the rate function to an explicit formula $\psi_{\alpha ,\delta}(2)=(1+\delta )^{\alpha}-1$.
The main challenge is to preserve the independence of edge-weights, and the simplest way is by rearranging the edges without adding or deleting them. We proceed in the following steps (see Proposition \ref{reduction of cliques} for details):
\begin{enumerate}
\item Split each vertex $v$ into $v_+, v_-$ attached to outward and inward edges, respectively. This process reduces all cliques, both-sided edges, and self-loops (i.e., a directed edge $(v,v)$ connecting a vertex $v$ to itself) while preserving the spectral norm and tree-like structure of each connected component.
\item By step (1), all directed edges are one-sided and we may consider the undirected network with the same `shape'. This undirected network has spectral norm greater than or equal to the original one, and its adjacency matrix is the symmetrization (Definition \ref{symmetrization}) of the original one.
\end{enumerate}

Step (1) automatically reduces all triangles, and thus all cliques as shown in Figure \ref{figure1}.

\begin{figure}[h] 
\centering
\begin{tikzpicture}[thick,scale=0.8, every node/.style={scale=0.8}]
    \draw (90:1) node[above=1mm] {$1$} -- (210:1) node[below left] {$2$} -- (330:1) node[below right] {$3$} -- cycle;
    \begin{scope}[xshift=4cm]
        \node (1) at (90:1) {}; \fill (90:1) circle (1pt) node[above=1mm] {$1$};
        \node (2) at (210:1) {}; \fill (210:1) circle (1pt) node[below left] {$2$};
        \node (3) at (330:1) {}; \fill (330:1) circle (1pt) node[below right] {$3$};
        \draw[<->] (1) -- (2);
        \draw[<->] (2) -- (3);
        \draw[<->] (3) -- (1);
        \draw[-latex, very thick] (1.3, .2) -- (2, .2);
    \end{scope}
    \begin{scope}[xshift=8cm, yshift=.2cm]
        \node (1+) at (115:1) {}; \fill (115:1) circle (1pt) node[above left] {\(1\text{\makebox[0pt][l]{$+$}}\)};
        \node (1-) at (65:1) {}; \fill (65:1) circle (1pt) node[above right] {\(1\text{\makebox[0pt][l]{$-$}}\)};
        \node (2+) at (-5:1) {}; \fill (-5:1) circle (1pt) node[right] {$2+$};
        \node (2-) at (-55:1) {}; \fill (-55:1) circle (1pt) node[below right] {\(2\text{\makebox[0pt][l]{$-$}}\)};
        \node (3+) at (235:1) {}; \fill (235:1) circle (1pt) node[below left] {\(3\text{\makebox[0pt][l]{$+$}}\)};
        \node (3-) at (185:1) {}; \fill (185:1) circle (1pt) node[left] {$3-$};
        \draw[->] (1+) -- (3-);
        \draw[->] (1+) -- (2-);
        \draw[->] (2+) -- (3-);
        \draw[->] (2+) -- (1-);
        \draw[->] (3+) -- (1-);
        \draw[->] (3+) -- (2-);
    \end{scope}
\end{tikzpicture}
\caption{The first figure shows an undirected triangle. The second figure is a directed triangle with all edges present, which is reduced to the third graph without cliques.}
\label{figure1}
\end{figure}
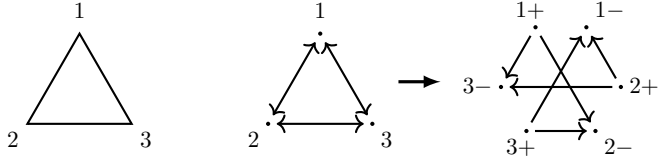
This clique reduction process is effective exclusively for directed networks. For if we consider each undirected edge $\overline{ij}$ as two directed edges $\overrightarrow{ij}$ and $\overrightarrow{ji}$, the edge-weights $w_{ij}$ and $w_{ji}$ must be equal. Then tail estimates for the spectral norm involve the sum of non-independent variables, and even the derivation of LLN results are impossible.

\subsection{Light-tailed weights}

For light-tailed weights, we follow the proof of the analogous problem for undirected networks by replacing undirected stars by directed stars, thus obtaining the same LLN and LDP results as in \cite{ghn22}. The underlying digraph $X$ is split into $X^{(1)}$ and $X^{(2)}$, based on truncation of edge-weights.
Under a high probability event, $X^{(1)}$ is split into a group of vertex-disjoint directed stars and spectrally negligible components with relatively small degree. Typically there are $n^{1-\gamma}$ directed stars with degree approximately $\gamma \frac{\log n}{\log \log n}$. The upper bound is established by calculating the contribution of these directed stars using Weibull tail estimates and Lemma \ref{directed star}. The lower bound is treated in a simpler way by placing a column vector with $\gamma_{\delta}\frac{\log n}{\log \log n}$ nonzero entries for $\gamma_{\delta}:=(1+\delta )^2(1-\frac{2}{\alpha})$.

\subsection{Heavy-tailed weights}

Applying the clique reduction technique (Proposition \ref{reduction of cliques}) to each connected component of $Z^{(1)}$ (the network corresponding to $X^{(1)}$), we reformulate the problem into estimating the spectral norm of undirected networks with maximum clique size $2$. Then the upper tail rate function is obtained by setting the maximum clique size $k=2$ in \eqref{chicken}, the upper tail rate function of \cite{ghn22}. The lower tail LDP is governed by the emergence of a single large edge-weight.

\section{Properties of graphs and digraphs}

\subsection{Structure of subcritical Erd\H{o}s-R\'enyi graphs}

Here we recall some properties of subcritical Erd\H{o}s-R\'enyi graphs $\mathcal{G}(n,q)$. As in \cite{ghn22}, we consider the highly subcritical regime 
\begin{align}
q \leq \frac{d}{n(\log n)^{\varepsilon}}
\end{align}
for some constants $d,\varepsilon >0$. Let $G$ be a graph such that $G\sim \mathcal{G}(n,q)$.

\begin{lemma}[{\cite[Lemma 5.3]{gn22}}] \label{d}
For $\delta_1 >0$, let $\mathcal{D}_{\delta_1}$ be an event defined by
\begin{align}
\mathcal{D}_{\delta_1} := \left\lbrace d_1 (G)\leq (1+\delta_1 )\frac{\log n}{\log \log n}\right\rbrace .
\end{align}
Then
\begin{align*}
\lim_{n\to\infty} \frac{-\log \mathbb{P}(\mathcal{D}_{\delta_1}^c)}{\log n}\geq \delta_1 .
\end{align*}
\end{lemma}

Let $C_1,\ldots ,C_{\ell}$ be the connected components of $G$. 

\begin{lemma}[{\cite[Lemma 5.4]{gn22}}, {\cite[Lemma 4.9]{ghn22}}] \label{c}
For $\delta_2 >0$, let $\mathcal{C}_{\varepsilon ,\delta_2}$ be the following event.
\begin{align}
\mathcal{C}_{\varepsilon ,\delta_2} := \left\lbrace \max_{1\leq i\leq \ell} |C_i| \leq \frac{1+\delta_2}{\varepsilon}\frac{\log n}{\log \log n} \right\rbrace .
\end{align}
Then,
\begin{align*}
\liminf_{n\to\infty}\frac{-\log \mathbb{P}(\mathcal{C}_{\varepsilon ,\delta_2}^c)}{\log n}\geq \delta_2 .
\end{align*}
\end{lemma}

\begin{lemma}[{\cite[Lemma 5.6]{gn22}}, {\cite[Lemma 4.10]{ghn22}}] \label{e}
For $\delta_3 >0$, let $\mathcal{E}_{\delta_3}$ be the event defined by
\begin{align}
\mathcal{E}_{\delta_3} := \left\lbrace \max_{1\leq i\leq \ell} \lbrace |E(C_i)| -|V(C_i)|\rbrace \leq \delta_3 \right\rbrace .
\end{align}
Then,
\begin{align*}
\liminf_{n\to\infty}\frac{-\log \mathbb{P}(\mathcal{E}_{\delta_3}^c)}{\log n} \geq \delta_3 .
\end{align*}
\end{lemma}

\subsection{Spectral properties of digraphs}

Here we provide methods to control the spectral and structural properties of directed networks using the results known for undirected networks.
We begin with a formal definition for associating directed networks with their undirected, or symmetrized versions which will be used multiple times throughout this paper.

\begin{definition} \label{symmetrization}
\upshape{
For an $n\times n$ matrix $A$, define its \textit{symmetrization} $\tilde{A}$ by
\begin{align*}
\tilde{A}_{ij}=\begin{cases}
A_{ij}\vee A_{ji} & \text{if $i\neq j$}\\
0 & \text{if $i=j$}
\end{cases}
\end{align*}
for $1\leq i,j\leq n$. The symmetrization $\tilde{G}$ of a directed network $G=(V,E,A)$ is defined as the undirected network corresponding to the adjacency matrix $\tilde{A}$.
}
\end{definition}

Especially, if $G$ has no self-loops and all edges are present in only one direction, then $\tilde{G}$ is formed simply by `ignoring direction' in $G$. Note that in the same situation, if the adjacency matrix of $G$ has i.i.d. entries, then the same holds for the adjacency matrix of $\tilde{G}$.

We now introduce a technique that decomposes cliques while not decreasing the spectral norm. This process reduces all cliques, both-sided edges, and self-loops in a directed network while preserving its tree-like structure.

\begin{proposition} \label{reduction of cliques}
Let $W$ be a connected directed network with nonnegative edge-weights. Then there exists an undirected network $H$ that satisfies the following properties:
\begin{enumerate}
\item $|E(H)|=|E(W)|$.
\item $|V(H)|\leq 2|V(W)|$.
\item $d_1(H)\leq d_1(W)$.
\item The size of the maximum clique in $H$ is $2$ and the edge-weights of $H$ are just a rearrangement of those of $W$.
\item Suppose $S(W)$ is the number of self-loops in $W$. 
If $H_1,\ldots ,H_m$ are the components of $H$, then
\begin{align*}
|E(H_i)|-|V(H_i)|\leq 2(|E(\tilde{W})|-|V(\tilde{W})|)+S(W)
\end{align*}
for $1\leq i\leq m$.
\item ${\displaystyle \| W\| \leq \| H\| =\max_{1\leq i\leq m}\| H_i\|}$.
\end{enumerate}
\end{proposition}

\begin{proof}
We begin with the construction of $H$. For each $v\in V(W)$, if $v$ has only inward (outward) edges, rename $v$ by $v_-$ ($v_+$). If $v$ has both inward and outward edges, split the point $v$ into $v_+$ and $v_-$ so that $v_+$ is connected to outward edges and $v_-$ to inward edges. A self-loop at $v$ becomes a directed edge from $v_+$ to $v_-$. Define $W'$ as the network obtained by applying this process for all $v\in V(W)$. Finally define $H=\tilde{W}'$.

\begin{figure}[h]
\centering
\begin{tikzpicture}[font={\fontsize{8pt}{12}\selectfont}]
    \begin{scope}
        \matrix[row sep={1cm,between origins}, column sep={1cm,between origins}]{
            & \bnode[left]{a}{$a$} & & & \\
            \bnode[left]{b}{$b$}
            & \bnode[above left]{c}{$c$}
            & \bnode[above left]{d}{$d$}
            & \bnode[above]{e}{$e$}
            & \bnode[above]{f}{$f$}\\
            \bnode[left]{g}{$g$}
            & \bnode[below]{h}{$h$}
            & \bnode[below]{i}{$i$}
            & \bnode[below]{j}{$j$}\\
        };
        \draw[->] (b) -- (c);
        \draw[->] (c) -- (a);
        \draw[->] (c) -- (h);
        \draw[->] (h) -- (g);
        \draw[->] (h) -- (i);
        \draw[->] (j) -- (i);
        \draw[->] (h) -- (d);
        \draw[->] (e) -- (j);
        \draw[->] (e) -- (f);
        \draw[->] (j) -- (d);
        \draw[<->] (c) -- (d);
        \draw[<->] (d) -- (i);
        \draw[<->] (c) -- (i);
        \draw[->] (d) edge[in=0,out=90,loop] ();
    \end{scope}
    \draw[-latex, very thick] (2.5, 0) -- (3.2, 0);
    \begin{scope}[xshift=5.9cm]
        \matrix[row sep={1cm,between origins}, column sep={1cm,between origins}]{
            &[-.333cm] \bnode[above]{a-}{$a_-$} &[-.667cm] & &[-.667cm] & & \\[-.333cm]
            \bnode[left]{b+}{$b_+$}
            & & \bnode[above]{c-}{$c_-$}
            & \bnode[above]{d+}{$d_+$}\\[-.667cm]
            & \bnode[left]{c+}{$c_+$}
            & & & \bnode[right, yshift=1pt]{d-}{$d-$}
            & \bnode[above]{e+}{$e_+$}
            & \bnode[above]{f-}{$f_-$}\\
            & \bnode[left]{h-}{$h_-$}
            & & & \bnode[right]{i+}{$i_+$}
            & \bnode[right]{j-}{$j_-$}\\[-.667cm]
            \bnode[left]{g-}{$g_-$}
            & & \bnode[below]{h+}{$h_+$}
            & \bnode[below]{i-}{$i_-$}
            & & \bnode[below]{j+}{$j_+$}\\
        };
        \draw[->] (c+) -- (a-);
        \draw[->] (b+) -- (c-);
        \draw[->] (c+) -- (h-);
        \draw[->] (h+) -- (g-);
        \draw[->] (h+) -- (i-);
        \draw[->] (j+) -- (i-);
        \draw[->] (d+) -- (i-);
        \draw[->] (d+) -- (c-);
        \draw[->] (i+) -- (d-);
        \draw[->] (j+) -- (d-);
        \draw[->] (e+) -- (f-);
        \draw[->] (e+) -- (j-);
        \draw[->] (c+) -- (d-);
        \draw[->] (c+) -- (i-);
        \draw[->] (i+) -- (c-);
        \draw[->] (d+) edge[bend left=45, looseness=1.5] (d-);
    \end{scope}
    \draw[-latex, very thick] (8.7, 0) -- (9.4, 0);
    \begin{scope}[xshift=11.4cm]
        \matrix[row sep={1cm,between origins}, column sep={1cm,between origins}]{
            & & & \bnode[left]{a-}{$a_-$} & \\
            \bnode[left]{b+}{$b_+$}
            & \bnode[above]{c-}{$c_-$}
            & \bnode[above]{i+}{$i_+$}
            & \bnode[left]{c+}{$c_+$}
            & \bnode[above]{h-}{$h_-$}\\
            & \bnode[left]{d+}{$d_+$}
            & \bnode[above left]{d-}{$d_-$}
            & \bnode[right]{j+}{$j_+$}\\
            & \bnode[below]{g-}{$g_-$}
            & \bnode[below]{h+}{$h_+$}
            & \bnode[below right]{i-}{$i_-$}\\
            & \bnode[below]{j-}{$j_-$}
            & \bnode[below]{e+}{$e_+$}
            & \bnode[below]{f-}{$f_-$}\\
        };
        \draw[->] (b+) -- (c-);
        \draw[->] (i+) -- (c-);
        \draw[->] (d+) -- (c-);
        \draw[->] (d+) -- (d-);
        \draw[->] (i+) -- (d-);
        \draw[->] (h+) -- (d-);
        \draw[->] (h+) -- (g-);
        \draw[->] (h+) -- (i-);
        \draw[->] (j+) -- (d-);
        \draw[->] (j+) -- (i-);
        \draw[->] (c+) -- (a-);
        \draw[->] (c+) -- (h-);
        \draw[->] (c+) -- (d-);
        \draw[->] (e+) -- (j-);
        \draw[->] (e+) -- (f-);
        \draw[->] (c+) edge[bend left=60, looseness=1.2] (i-);
        \draw[->, rounded corners] (d+) -- +(-.333cm, -.667cm) -- +(-.333cm, -1.667cm) -- +(2cm, -1.667cm) -- (i-);
    \end{scope}
\end{tikzpicture}
\caption{Applying the clique reduction to the first figure yields the second figure. The last figure is a redrawing of the second one which shows that no cliques exist.}
\end{figure}
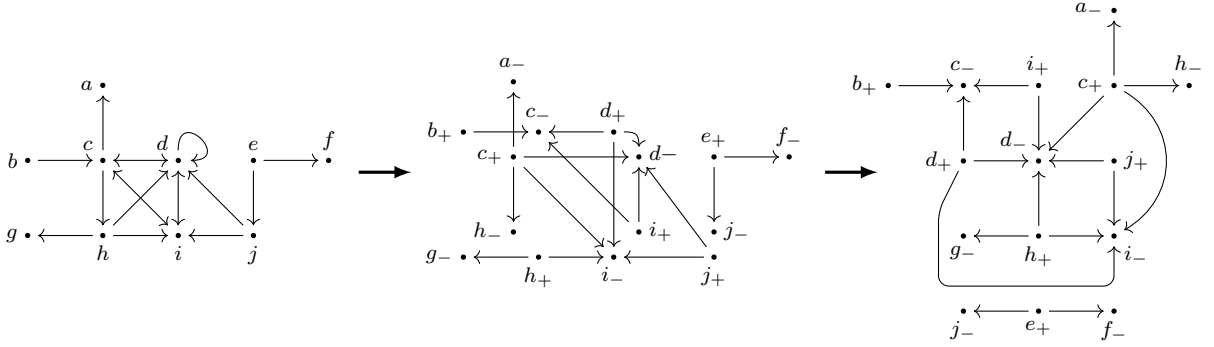

Since the construction of $H$ is a simple rearrangement of edges and every vertex is left unchanged or duplicated, properties (1) to (3) are clear. Let $V_+(W')$ and $V_-(W')$ be the set of vertices designated as $+$ and $-$, respectively. Then every edge in $E(W')$ is $\overrightarrow{xy}$ for some $x\in V_+(W')$ and $y\in V_-(W')$.
Suppose that $W'$ contains a triangle $xyz$ and $xy$ has direction $\overrightarrow{xy}$. Then $y$ has only inward edges, so $yz$ has direction $\overrightarrow{zy}$. This means that $x,z\in V_+(W')$, so there cannot be a directed edge between $x$ and $z$, which is a contradiction. Then $W'$ and thus $H$ does not contain a triangle which implies (4).

Since $W'$ has no self-loops and both-sided edges, the number of edges and vertices of $W'$ and $H$ coincide along with their undirected structure. Thus it suffices to prove property (5) for $W'$ and its components $W_1',\ldots ,W_m'$ corresponding to $H_1,\ldots ,H_m$, respectively. We first consider the case when $W'$ is connected.
If $W$ has no both-sided edges, then the property is clear from $|E(W')|=|E(\tilde{W})|+S(W)$ and $|V(W')|\geq |V(W)|=|V(\tilde{W})|$. Suppose there exists a both-sided edge in $W$, and consider the subgraph $D$ of $\tilde{W}$ which corresponds to both-sided edges of $W$. Note that since any vertex $d\in V(D)$ is attached to a both-sided edge, it readily splits into $d_+$ and $d_-$ in $W'$. We will add edges to $W$ until $D$ grows up to $\tilde{W}$ while not decreasing $|E(W')|-|V(W')|$ for the corresponding $W'$. Let $d$ be a vertex in $V(D)$ and $v$ be adjacent to $d$ in $W$ but $\overline{dv}\notin E(\tilde{W})$. Then there is a one-sided edge between $d$ and $v$, and adding the other side increases $|E(W')|$ by $1$. If $v\in V(D)$, both $v$ and $d$ splits into two points in $W'$ before adding the edge, so $|V(W')|$ is unchanged. If $v\notin V(D)$, only $v$ newly splits into two in $W'$, so $|V(W')|$ increases by $1$. Applying this process inductively for all adjacent edges, $D$ grows up to $\tilde{W}$ while not increasing $|E(W')|-|V(W')|$. When $D=\tilde{W}$, all edges in $W$ are both-sided and all vertices are split into two in $W'$, so $|E(W')|=2|E(\tilde{W})|+S(W)$ and $|V(W')|=2|V(W')|$. This completes the proof for connected $W'$. Generally, a connected component $W_i'$ of $W'$ corresponds to a connected subgraph $G$ of $W$ such that $G'=W_i'$. Then $\tilde{G}$ is a connected subgraph of $\tilde{W}$ and Lemma \ref{e-v lemma} below implies $|E(\tilde{G})|-|V(\tilde{G})|\leq |E(\tilde{W})|-|V(\tilde{W})|$. Thus
\begin{align*}
|E(W_i')|-|V(W_i')| \leq 2(|E(\tilde{G})|-|V(\tilde{G})|)+S(G)\leq 2(|E(\tilde{W})|-|V(\tilde{W})|)+S(W).
\end{align*}

Since $H_i$ are vertex disjoint, the second equality of (6) is clear. For the first equality, let $W=(w_{ij})_{1\leq i,j\leq n}$. We will view $w_{ij}$ as the weight of both $\overrightarrow{ij}$ in $W$ and $\overrightarrow{i_+j_-}$ in $W'$. Since the weights are nonnegative and $w_{ij}$ is nonzero only for $j\in V_-(W')$,
\allowdisplaybreaks
\begin{align*}
\| W\|^2 & =\sup_{f\in [0,\infty )^n, \| f\|_2 =1}\sum_{i=1}^{n}\left( \sum_{j=1}^{n}w_{ij}f_j\right)^2 =\sup_{f\in [0,\infty )^n, \| f\|_2 =1}\sum_{1\leq j,k\leq n}\left( \sum_{i=1}^{n}w_{ij}w_{ik}\right) f_j f_k\\
& =\sup_{f\in [0,\infty )^n, \| f\|_2 =1}\sum_{j,k\in V_-(W')}\left( \sum_{i\in V_+(W')}w_{ij}w_{ik}\right) f_j f_k\\
& = \sup_{\tilde{f}\in [0,\infty )^{|V(W')|}, \| \tilde{f}\|_2 =1}\sum_{j,k\in V(W')}\left( \sum_{i\in V(W')}w_{ij}w_{ik}\right) \tilde{f}_j\tilde{f}_k  =\| W'\|^2.
\end{align*}
Here $f$ is indexed by $[n]$, and $\tilde{f}$ by the vertex set $V(W')$ itself.
Since $W'$ has no self-loops, the matrix inequality $W'\leq H$ holds entrywise and thus $\| W\| =\| W'\| \leq \| H\|$.
\end{proof}

Now we prove the fact that the value of $|E|-|V|$ for subgraphs of connected graphs does not exceed that of the original graph which was used in the main proof above.

\begin{lemma} \label{e-v lemma}
Suppose $K$ is a subgraph of a connected undirected graph $G$.
Then $|E(K)|-|V(K)|\leq |E(G)|-|V(G)|$.
\end{lemma}

\begin{proof}
For a graph $G$ with $p=p(G)$ components, the nullity is defined as $N[G]=(|E(G)|-|V(G)|)+p(G)$. Let $K_s$ be the graph consisting of edges in $G\setminus K$ and the corresponding vertices. \cite[Theorem 1]{b66} states that $N[G]=N[K]+N[K_s]$. For a spanning tree $T$ of $K_s$, 
\begin{align*}
N[K_s]=|E(K_s)|-|V(K_s)|+p(K_s)\geq |E(T)|-|V(T)|+1=0,
\end{align*}
and thus $N[G]\geq N[K]$. Since $G$ is connected, $p(G)=1 \leq p(K)$ and the lemma follows.
\end{proof}

Finally we calculate the norm of weighted directed stars for later use.

\begin{lemma} \label{directed star}
Suppose that $S$ is a weighted directed star graph with adjacency matrix
\begin{align*}
A =\begin{pmatrix}
0 & a_1 & \cdots & a_{n-1} \\
b_1 &  &  &  \\
\vdots &  &  & \\
b_{m-1} & & &
\end{pmatrix}.
\end{align*}
Then
\begin{align*}
\| A\| =\left( \sum_{i=1}^{n-1}a_i^2\vee \sum_{i=1}^{m-1}b_i^2 \right)^{\frac{1}{2}}.
\end{align*}
\end{lemma}

\begin{proof}
For a vector $v=\begin{pmatrix}
v_0 & \cdots & v_{n-1}
\end{pmatrix}^{\intercal}$ with $\| v\|_2 =1$,
\begin{align*}
\| Av\|_2^2 & = \left( \sum_{i=1}^{n-1}a_i v_i\right)^2 +\sum_{i=1}^{m-1}b_i^2 v_0^2 \leq \sum_{i=1}^{n-1}a_i^2\sum_{i=1}^{n-1}v_i^2 +v_0^2\sum_{i=1}^{m-1}b_i^2 =(1-v_0^2)\sum_{i=1}^{n-1}a_i^2 +v_0^2 \sum_{i=1}^{m-1}b_i^2\\
& \leq \sum_{i=1}^{n-1}a_i^2\vee \sum_{i=1}^{m-1}b_i^2 .
\end{align*}
Equality holds when $v_0=0$ and $v_1,\ldots ,v_{n-1}$ satisfies the equality conditions for the Cauchy-Schwarz inequality, or when $v_0=1$, depending on the values of $a_i$ and $b_i$.
\end{proof}

\section{Light-tailed weights}

Here we consider the light-tailed case $\alpha >2$. Recall that 
\begin{align*}
\lambda_{\alpha}^{\mathrm{light}} := 2^{\frac{1}{\alpha}}\alpha^{-\frac{1}{2}}(\alpha -2)^{\frac{1}{2}-\frac{1}{\alpha}}\frac{(\log n)^{\frac{1}{2}}}{(\log \log n)^{\frac{1}{2}-\frac{1}{\alpha}}} .
\end{align*}
We will simply write $\lambda_{\alpha}^{\mathrm{light}}$ as $\lambda_{\alpha}$ throughout this section and let
\begin{align}
B_{\alpha} := 2^{\frac{1}{\alpha}}\alpha^{-\frac{1}{2}}(\alpha -2)^{\frac{1}{2}-\frac{1}{\alpha}}.
\end{align}

\subsection{Upper tail estimate}

We first recall the upper tail estimate that will be proven in this section:

\ltut*

Define
\begin{align} \label{gamma delta}
\gamma_{\delta} := (1+\delta )^2 \left( 1-\frac{2}{\alpha}\right) .
\end{align}
The governing structure for the upper tail of $\| Z\|$ will be a directed star of degree $\lceil \gamma_{\delta}\frac{\log n}{\log \log n}\rceil$ with large edge-weights. Note that an undirected star with the same properties played the same role for the largest eigenvalue of undirected networks in \cite{ghn22}.

\subsubsection{Lower bound for the upper tail} \label{lower bound for the upper tail}

The proof generally follows the corresponding undirected case in \cite{ghn22}, but is simplified by replacing stars by column vectors.
For the lower bound, instead of considering a directed star, it suffices to consider a column vector with $\lceil \gamma_{\delta}\frac{\log n}{\log \log n}\rceil$ nonzero coordinates. We split the proof into two cases depending on the value of $\gamma_{\delta}$.

\emph{Case 1: ${\gamma_{\delta}<1}$.} For $0<\gamma ,\rho <1$, let $\mathcal{A}_{\gamma ,\rho}$ be the event where $X$ has $m=\lceil n^{1-\gamma -\rho }\rceil$ columns each with $\lceil \gamma \frac{\log n}{\log \log n}\rceil$ nonzero coordinates.
By Lemma \ref{binomial loglog estimate}, the probability that a single column has at least $\lceil \gamma_{\delta} \frac{\log n}{\log \log n}\rceil$ nonzero coordinates is bounded by
\begin{align} \label{single column}
\mathbb{P} \left( \mathrm{Binom}\left( n,\frac{d}{n}\right) \geq \gamma_{\delta}\frac{\log n}{\log \log n}\right) \leq n^{-\gamma_{\delta}+w}
\end{align}
for some $w=o(1)$. Using this,
\begin{align*}
\mathbb{P} (\mathcal{A}_{\gamma_{\delta},\rho}) & \geq \mathbb{P}\left( \mathrm{Binom}\left( n, n^{-\gamma_{\delta}+w}\right) \geq n^{1-\gamma_{\delta}-\rho} \right) \geq 1-e^{-nI_{n^{-\gamma_{\delta}+w}}(n^{-\gamma_{\delta} -\rho})}.
\end{align*}
Since $n^{-\gamma_{\delta}-\rho} \leq \frac{1}{2}n^{-\gamma_{\delta}+w}$ for large $n$, by Lemma \ref{entropy lower bound},
\begin{align*}
I_{n^{-\gamma_{\delta}+w}}(n^{-\gamma_{\delta}-\rho}) \geq cn^{-\gamma_{\delta }+w}
\end{align*}
holds for some constant $c$. Then
\begin{align} \label{32}
\mathbb{P}(\mathcal{A}_{\gamma_{\delta},\rho}) & \geq 1-e^{-n^{1-\gamma_{\delta}+o(1)}}\geq \frac{1}{2}
\end{align}
for large $n$. Let $X_{\cdot \sigma (1)},\ldots ,X_{\cdot \sigma (m)}$ be the columns with at least $\lceil \gamma \frac{\log n}{\log \log n}\rceil$ nonzero coordinates. Note that if $Z_a$ is a submatrix of $Z$, then $\| Z \| \geq \| Z_a \|$ by Lemma \ref{submatrix}, and the norm of a column vector $V=\begin{pmatrix}
v_1 & \cdots & v_n
\end{pmatrix}^{\intercal}$ is $\left( \sum_{i=1}^{n}v_i^2\right)^{\frac{1}{2}}$. Conditioned on the event $\mathcal{A}_{\gamma_{\delta},\rho}$, by the tail estimates for Weibull distributions in \eqref{151},
\begin{align*}
\mathbb{P}\left( \sum_{i=1}^{n}Z_{i\sigma (k)}^2 \geq (1+\delta )^2 \lambda_{\alpha}^2 \, | \, X\right) \geq n^{-(1+\delta )^2 \frac{2}{\alpha}+o(1)}.
\end{align*}
Therefore
\begin{align*}
\mathbb{P}\left( \max_{k=1,\ldots ,m}\sum_{i=1}^{n}Z_{i\sigma (k)}^2 \geq (1+\delta )^2 \lambda_{\alpha}^2 \, | \, X\right) & \geq  1-\left( 1-n^{-(1+\delta )^2 \frac{2}{\alpha}+o(1)}\right)^m \\
& \geq  1-e^{-n^{1-\gamma_{\delta}-\rho -(1+\delta )^2\frac{2}{\alpha}+o(1)}}  \geq \frac{1}{2}n^{1-(1+\delta )^2 -\rho +o(1)}
\end{align*}
and
\begin{align} \label{ltutlb1}
\mathbb{P} \left( \| Z\| \geq (1+\delta )\lambda_{\alpha}\right) & \geq \mathbb{E}\left[ \mathbb{P}\left( \| Z\| \geq (1+\delta )\lambda_{\alpha} \, | \, X\right) \mathds{1}_{\mathcal{A}_{\gamma_{\delta},\rho}}\right] \nonumber \\
& \geq \mathbb{P}(\mathds{1}_{\mathcal{A}_{\gamma_{\delta},\rho}}) \mathbb{P}\left( \max_{k=1,\ldots ,m}\sum_{i=1}^{n}Z_{i\sigma (k)}^2 \geq (1+\delta )^2 \lambda_{\alpha}^2 \, | \, X\right) \geq \frac{1}{4}n^{1-(1+\delta )^2 -\rho +o(1)}.
\end{align}

\emph{Case 2: $\gamma_{\delta}\geq 1$.} For $\rho >0$, let $\mathcal{A}_{\gamma}'$ be the event where $X$ has a column with at least $\lceil \gamma \frac{\log n}{\log \log n}\rceil$ nonzero coordinates. Condition on the event $\mathcal{A}_{(1+\rho )\gamma_{\delta}}'$ and let $X_{\cdot k}$ be the corresponding column. Again by the tail estimate \eqref{151},
\begin{align*}
\mathbb{P} \left( \| Z\| \geq (1+\delta )\lambda_{\alpha} \, | \, X\right) & \geq \mathbb{P}\left( \sum_{i=1}^{n}Z_{ik}^2 \geq (1+\delta )^2 \lambda_{\alpha}^2 \, | \, X\right) \\
& \geq n^{-(1+\delta )^{\alpha}\frac{2}{\alpha -2}(1-\frac{2}{\alpha})^{\frac{\alpha}{2}}((1+\rho )\gamma_{\delta})^{1-\frac{\alpha}{2}}+o(1)} =n^{-(1+\delta )^2 \frac{2}{\alpha}(1+\rho )^{1-\frac{2}{\alpha}}+o(1)}.
\end{align*}
From \eqref{single column},
\begin{align*}
\mathbb{P}(\mathcal{A}_{(1+\rho )\gamma_{\delta}}') =1-\left( 1-n^{-(1+\rho )\gamma_{\delta}+o(1)}\right)^n=n^{1-(1+\rho )\gamma_{\delta}+o(1)}.
\end{align*}
Thus
\begin{align} \label{ltutlb2}
\mathbb{P}(\| Z\| \geq (1+\delta )\lambda_{\alpha}) \geq n^{1-(1+\rho )\gamma_{\delta}-(1+\delta )^2 \frac{2}{\alpha}(1+\rho )^{1-\frac{2}{\alpha}}+o(1)}.
\end{align}
The theorem follows by letting $\rho \to 0$ in \eqref{ltutlb1} and \eqref{ltutlb2}.

\subsubsection{Upper bound for the upper tail}

We begin with a large deviation principle for the underlying Erd\H{o}s-R\'enyi digraph $X$.
For simplicity, denote $t_n=\frac{\log n}{\log \log n}$. First we restate a large deviation principle for the largest eigenvalue of an undirected Erd\H{o}s-R\'enyi graph:

\begin{lemma}[{\cite[Theorem 1.1]{bbg21}}] \label{graph eigenvalue}
Let $P$ be the adjacency matrix of the Erd\H{o}s-R\'enyi graph $\mathcal{G}(n,\frac{d}{n})$ where $d>0$ is a fixed constant.
Then for any $\delta >0$,
\begin{align*}
\lim_{n\to\infty}-\frac{\log \mathbb{P}\left( \lambda_1 (P)\geq (1+\delta )t_n^{\frac{1}{2}}\right)}{\log n}=(1+\delta )^2-1 .
\end{align*}
\end{lemma}

A similar bound for the norm of $X\sim \mathcal{G}_d(n,\frac{d}{n})$ can be derived using the undirected case:

\begin{lemma} \label{graph ldp}
For any $\delta >0$,
\begin{align*}
\lim_{n\to\infty}-\frac{\log \mathbb{P}\left( \| X\| \geq (1+\delta )t_n^{\frac{1}{2}}\right)}{\log n}\geq (1+\delta )^2-1 .
\end{align*}
\end{lemma}

\begin{proof}
Consider the symmetric matrix $\tilde{X}$ defined as in Definition \ref{symmetrization}. Then $\tilde{X}_{ij}$ for $i>j$ are i.i.d with distribution $\mathrm{Bernoulli}(2p-p^2)$ and $X\leq \tilde{X}+I_n$ where the inequality holds entrywise and $I_n$ is the $n\times n$ identity matrix. Since $\| I_n\| =1$, this implies
\begin{align*}
\mathbb{P}\left( \| X\| \geq (1+\delta )t_n^{\frac{1}{2}}\right) \leq \mathbb{P}\left( \| \tilde{X}\| \geq (1+\delta )t_n^{\frac{1}{2}}-1\right) =\mathbb{P} \left( \lambda_1 (\tilde{X}) \geq (1+\delta )t_n^{\frac{1}{2}}-1\right) .
\end{align*}
Since $\tilde{X}$ is the adjacency matrix of $\mathcal{G}(n,2p-p^2)$ which is stochastically dominated by $\mathcal{G}(n,2p)$, applying Lemma \ref{graph eigenvalue} gives the desired result.
\end{proof}

Now consider the decomposition 
\begin{align*}
Y_{ij}=Y_{ij}^{(1)}+Y_{ij}^{(2)}
\end{align*}
where $Y_{ij}^{(1)}=Y_{ij}\mathds{1}_{|Y_{ij}|>(\varepsilon \log \log n)^{\frac{1}{\alpha}}}$ and $Y_{ij}^{(2)}=Y_{ij}\mathds{1}_{|Y_{ij}|\leq (\varepsilon \log \log n)^{\frac{1}{\alpha}}}$. Also write $Z=Z^{(1)}+Z^{(2)}$ where $Z_{ij}^{(1)}=X_{ij}Y_{ij}^{(1)}$ and $Z_{ij}^{(2)}=X_{ij}Y_{ij}^{(2)}$. Since $|Z_{ij}^{(2)}|\leq X_{ij}(\varepsilon \log \log n)^{\frac{1}{\alpha}}$, Lemma \ref{graph ldp} implies the following:

\begin{lemma} \label{light tail negligible}
For any $\delta >0$,
\begin{align*}
\liminf_{n\to\infty} \frac{-\log \mathbb{P}\left( \| Z^{(2)}\| \geq \varepsilon^{\frac{1}{\alpha}}(1+\delta )\frac{\lambda_{\alpha}}{B_{\alpha}}\right)}{\log n} \geq (1+\delta )^2-1 .
\end{align*}
\end{lemma}

By the tail behavior of Weibull distributions, $X^{(1)}$ is distributed as $\mathcal{G}_d(n,q)$ where
\begin{align*}
q\leq \frac{d'}{n(\log n)^{\varepsilon}}
\end{align*}
for some constant $d'>0$. Then the symmetric matrix $\tilde{X}^{(1)}$ is distributed as $\mathcal{G}(n,2q-q^2)$, which is stochastically dominated by $\mathcal{G}(n,\frac{2d'}{n(\log n)^{\varepsilon}})$. Thus the structural lemmas from Section 3 can be applied to $\tilde{X}^{(1)}$. 

For further decomposition of $\tilde{X}^{(1)}$ we use the following lemma.

\begin{lemma}[{\cite[Lemma 5.4]{ghn22}}] \label{star decomposition}
There exists an event $\mathcal{W}$ measurable with respect to $\tilde{X}^{(1)}$ with probability at least $1-e^{-\omega (\log n)}$ under which $\tilde{X}^{(1)}$ can be decomposed into a graph $\tilde{X}_1^{(1)}$ which is a vertex-disjoint union of stars, and a graph $\tilde{X}_2^{(1)}$ whose maximum degree is $o\left( \frac{\log n}{\log \log n}\right)$.
\end{lemma}

We condition on the high-probability event $\mathcal{W}$. Define $X_1^{(1)}$, $X_2^{(1)}$ as the directed subgraphs of $X^{(1)}$ corresponding to $\tilde{X}_1^{(1)}$, $\tilde{X}_2^{(1)}$, respectively. Here we have the understanding that $X^{(1)}=X_1^{(1)}+X_2^{(1)}$ where all self-loops of $X^{(1)}$ are included in $X_2^{(1)}$. In this way, Lemma \ref{star decomposition} implies that $X_1^{(1)}$ is a union of vertex-disjoint directed stars, and $X_2^{(1)}$ has maximum degree $o\left( \frac{\log n}{\log \log n}\right)$ since a self-loop increases the degree by only $1$. Similarly, let $Z_1^{(1)}$ and $Z_2^{(1)}$ be the decomposition of $Z^{(1)}$ corresponding to $X_1^{(1)}$ and $X_2^{(1)}$, respectively.

The following proposition is a directed version of \cite[Proposition 5.5]{ghn22} and will be crucial in analyzing the spectral behavior of $Z_2^{(1)}$.

\begin{proposition} \label{infinity}
For $n\in\mathbb{N}$ and $\varepsilon ,\delta_1, \delta_2, \delta_3 >0$, let $\mathcal{G}:= \mathcal{G}_{n,\varepsilon ,\delta_1 ,\delta_2, \delta_3}$ be the set of connected directed networks $G$ such that
\begin{enumerate}
\item $d_1(G) =o\left( \frac{\log n}{\log \log n}\right)$,
\item $|V(G)| \leq \frac{1+\delta_1}{\varepsilon}\frac{\log n}{\log \log n}$,
\item $|E(\tilde{G})| \leq |V(\tilde{G})| +\delta_2$,
\item $S(G)\leq \delta_3$,
\end{enumerate}
where $\tilde{G}$ is the symmetrization of $G$ and $S(G)$ is the number of self-loops in $G$.
Assume that the edge-weights are i.i.d. Weibull distributions with shape parameter $\alpha >2$ conditioned to be greater than $(\varepsilon \log \log n)^{\frac{1}{\alpha}}$ in absolute value. Then, for any constant $c>0$,
\begin{align} \label{limit infinity}
\lim_{n\to\infty}\frac{-\log \sup_{G\in\mathcal{G}} \mathbb{P}(\| G\| \geq c\lambda_{\alpha})}{\log n}=\infty .
\end{align}
\end{proposition}

\begin{proof}
The strategy is to reduce the problem to undirected networks so that the original proposition for undirected networks with the same structural conditions can be applied. Since taking absolute values of edge-weights does not decrease the norm, it suffices to prove \eqref{limit infinity} for weights having absolute values of conditioned Weibull distributions.
Let $H^G$ be the undirected network obtained by applying Proposition \ref{reduction of cliques} to $G$ and let $H^G_1,\ldots ,H^G_{\xi (G)}$ be its components where $\xi (G)\in\mathbb{N}$ is the number of components of $G$. Then each $H^G_i$ ($1\leq i\leq \xi (G)$) satisfies the following structural properties:
\begin{gather*}
 d_1(H^G_i) \leq d_1(H^G) \leq d_1(G) =o\left( \frac{\log n}{\log \log n}\right) ,\\
 |V(H^G_i)| \leq |V(H^G)| \leq 2|V(G)| \leq \frac{2(1+\delta_1)}{\varepsilon}\frac{\log n}{\log \log n},\\
|E(H^G_i)|-|V(H^G_i)| \leq 2(|E(\tilde{G})|-|V(\tilde{G})|)+S(G)\leq 2\delta_2+\delta_3.
\end{gather*}
Now \cite[Proposition 5.5]{ghn22} applies to the family $\mathcal{H}=\lbrace H_i^G\, | \, G\in\mathcal{G}, 1\leq i\leq \xi (G)\rbrace$ and 
\begin{align} \label{symmetric infinity}
\lim_{n\to\infty}\frac{-\log\sup_{H_i^G\in \mathcal{H}} \mathbb{P}(\lambda_1 (H_i^G)\geq c\lambda_{\alpha})}{\log n}=\infty .
\end{align}
Since $H_i^G$ is an undirected network with nonnegative entries, its largest eigenvalue equals the spectral norm and 
\begin{align}
\| H^G\| =\max_{1\leq i\leq \xi (G)}\| H_i^G\| =\max_{1\leq i\leq \xi (G)}\lambda_1 (H_i^G).
\end{align}
Each $H^G$ has $o(n)$ components, so
\begin{align} \label{supsupsup}
\sup_{G\in\mathcal{G}}\mathbb{P}(\| G\| \geq c\lambda_{\alpha}) & \leq \sup_{G\in\mathcal{G}}\mathbb{P}(\| H^G\| \geq c\lambda_{\alpha}) \leq o(n)\sup_{H_i^G\in\mathcal{H}}\mathbb{P}(\lambda_1(H_i^G)\geq c\lambda_{\alpha}).
\end{align}
Applying \eqref{supsupsup} to \eqref{symmetric infinity} completes the proof.
\end{proof}

Next we estimate the spectral contribution of $Z_1^{(1)}$. We denote $g(\gamma )=\lceil \gamma \frac{\log n}{\log \log n}\rceil$ and $d(S)=d^+(v) \vee d^-(v)$ for a directed star $S$ and its root vertex $v$. 
First we have the following analogue of \cite[Lemma 5.6]{ghn22} for the norm of a single directed star:

\begin{lemma} \label{single directed star}
Suppose that $S$ is a weighted directed star such that $d(S)\leq g(\gamma )$ for some $\gamma >0$, with i.i.d. weights having Weibull distributions with shape parameter $\alpha >2$ and conditioned to be greater than $(\varepsilon \log \log n)^{\frac{1}{\alpha}}$ in absolute value. Then, for any $\rho >0$,
\begin{align*}
\liminf_{n\to\infty}\frac{-\log \mathbb{P}(\| S\| \geq (1+\rho )\lambda_{\alpha})}{\log n}\geq (1+\rho )^{\alpha}\frac{2}{\alpha -2}\left( 1-\frac{2}{\alpha}\right)^{\frac{\alpha}{2}}\gamma^{1-\frac{\alpha}{2}}-\varepsilon \gamma .
\end{align*}
\end{lemma}

\begin{proof}
Let $\lbrace \tilde{Y}_i\rbrace$ be i.i.d. Weibull distributions with shape parameter $\alpha >2$ and conditioned to be greater than $(\varepsilon \log \log n)^{\frac{1}{\alpha}}$ in absolute value. By Lemma \ref{directed star},
\begin{align*}
\mathbb{P}(\| S\| \geq (1+\rho )\lambda_{\alpha}) \leq 2\mathbb{P}\left( \tilde{Y}_1^2+\cdots +\tilde{Y}_{g(\gamma )}^2 \geq (1+\rho )^2 \lambda_{\alpha}^2\right) .
\end{align*}
The factor $2$ is negligible and applying the tail bound \eqref{152} with $d=1+\rho$ and $b=\gamma$ yields the desired estimate. 
\end{proof}

The contribution of directed stars will be analyzed using the fact that a directed star $S$ in $X_1^{(1)}$ corresponds to an undirected star $\tilde{S}$ in $\tilde{X}_1^{(1)}$. We first restate some definitions from \cite{ghn22}.
Let $d\left( \tilde{X}^{(1)},v\right)$ be the degree of $v$ in $\tilde{X}_1^{(1)}$ and define
\begin{align}
D_{\gamma}^{(1)} =\left\lbrace v\in V: d\left( \tilde{X}^{(1)},v\right) \geq g(\gamma )\right\rbrace 
\end{align}
for $\gamma \geq 0$. For a small constant $\kappa >0$ to be determined later, define $m$ as an integer satisfying $m\kappa <1\leq (m+1) \kappa$. Now define the following events measurable with respect to $\tilde{X}_1^{(1)}$:
\begin{align}
\mathcal{P}_{\kappa} & :=\left\lbrace |D_{i\kappa}^{(1)}| \leq n^{1-i\kappa +\kappa } \  \text{for all } i=0,1,\cdots ,m\right\rbrace ,\\
\mathcal{L}_{\delta ,\kappa} & :=\left\lbrace |D_{1+\kappa }^{(1)}| \leq \frac{(1+\delta )^2}{\kappa}\right\rbrace .
\end{align}
$\mathcal{P}_{\kappa}$ controls the number of vertices in each degree range under $g(\gamma )$, and $\mathcal{L}_{\delta ,\kappa}$ bounds the number of vertices with large degree. Using these definitions, we group directed stars according to the degree of their symmetrizations (undirected stars) and calculate the contribution from each group.

By observing the proof of \cite[Lemma 5.7]{ghn22}, we note that the same results follows for the spectral norm of directed stars in $X_1^{(1)}$ if we replace \cite[Lemma 5.6]{ghn22} by Lemma \ref{single directed star} above.

\begin{lemma}[{\cite[Lemma 5.7, 5.8]{ghn22}}] \label{modified lemma 5.7}
Let $\mathcal{S}$ be a collection of directed stars in $X_1^{(1)}$. For any $h>0$ and $\gamma =i\kappa <1$ or $\gamma \geq 1+\kappa$, define
\begin{align}
\sigma_{\max}(\gamma ,h) := \max_{S\in \mathcal{S},d(\tilde{S})\in (g(\gamma ), g(\gamma +h)]} \| S\|
\end{align}
if there exists a directed star satisfying the conditions in the subscript and $\sigma_{\max}(\gamma ,h) :=0$ otherwise. Then for any $\rho >0$,
\begin{align} \label{24}
\liminf_{n\to\infty} & \frac{-\log \mathbb{E}\left[ \mathbb{P}\left( \sigma_{\max}(\gamma ,h) \geq (1+\rho )\lambda_{\alpha} \, | \, X^{(1)}\right) \mathds{1}_{\mathcal{P}_{\kappa}\cap \mathcal{L}_{\delta ,\kappa}}\right]}{\log n}\nonumber\\
& \geq -f_{\alpha ,\rho}(\gamma +h) -h-\kappa -\varepsilon (\gamma +h),
\end{align}
where the function $f_{\alpha ,\rho}: (0,\infty )\to \mathbb{R}$ is defined by
\begin{align}
f_{\alpha ,\rho }(x) := 1-x-(1+\rho )^{\alpha}\frac{2}{\alpha -2}\left( 1-\frac{2}{\alpha}\right)^{\frac{\alpha}{2}}x^{1-\frac{\alpha}{2}}.
\end{align}
Furthermore, for $\alpha >2$,
\begin{align}
\max_{\gamma >0}f_{\alpha ,\rho}(\gamma ) =1-(1+\rho )^2 \quad \text{and} \quad \gamma_{\rho} := \underset{\gamma >0}{\arg\max}f_{\alpha ,\rho}(\gamma )=(1+\rho )^2 \left( 1-\frac{2}{\alpha}\right) .
\end{align}
\end{lemma}
In \eqref{24}, the main contribution is $-f_{\alpha ,\rho}(\gamma +h)$ and $-h-\kappa -\varepsilon (\gamma +h)$ is considered as an error term.

Finally, we restate some lemmas required to estimate the probabilities of the events $\mathcal{P}_{\kappa}$ and $\mathcal{L}_{\delta ,\kappa}$.
\begin{lemma}[{\cite[Proposition 4.6]{ghn22}}] \label{number of large degree}
For any $0<\kappa <1$, $\mu >0$ and an integer $m$ that satisfies $m\kappa <1\leq (m+1)\kappa$, we have
\begin{align*}
\mathbb{P}(| D_{i\kappa}^{(1)}|\leq n^{1-i\kappa +\kappa} \text{ for all $i=0,1,\cdots ,m$}) \geq 1-n^{-\mu \kappa}
\end{align*}
for sufficiently large $n$.
\end{lemma}
To estimate the probability of the event $\mathcal{L}_{\delta ,\kappa}$ we utilize the following lemma on the largest degrees of Erd\H{o}s-R\'enyi graphs.
\begin{lemma}[{\cite[Proposition 1.3]{bbg21}}] \label{degrees}
Let $d_s$ be the $s$-th largest degree of $\mathcal{G}(n,\frac{d}{n})$. Then
\begin{align*}
\lim_{n\to\infty}\frac{-\log \mathbb{P}(d_1 \geq (1+\delta_1 )t_n, \ldots ,d_r \geq (1+\delta_r )t_n)}{\log n}=\sum_{s=1}^{r}\delta_s .
\end{align*}
\end{lemma}
Applying Lemma \ref{degrees} with $\delta_1 =\cdots =\delta_{\frac{(1+\delta )^2}{\kappa}}=\kappa$ yields
\begin{align} \label{l bound}
\mathbb{P}\left( \mathcal{L}_{\delta ,\kappa}^c\right) \leq n^{-(1+\delta )^2 +o(1)}.
\end{align}

\noindent \emph{Proof of upper bound of the upper tail.} Since $\| Z\| \leq \| Z^{(1)}\| +\| Z^{(2)}\|$,
\begin{align} \label{light tail split}
\mathbb{P}(\| Z\| \geq (1+\delta )\lambda_{\alpha}) \leq \mathbb{P}\left( \| Z^{(1)}\| \geq (1+\delta )\left( 1-\frac{\varepsilon^{\frac{1}{\alpha}}}{B_{\alpha}}\right) \lambda_{\alpha}\right) +\mathbb{P} \left( \| Z^{(2)}\| \geq \varepsilon^{\frac{1}{\alpha}}(1+\delta )\frac{\lambda_{\alpha}}{B_{\alpha}}\right) .
\end{align}
Lemma \ref{light tail negligible} implies
\begin{align} \label{37}
\mathbb{P} \left( \| Z^{(2)}\| \geq \varepsilon^{\frac{1}{\alpha}}(1+\delta )\frac{\lambda_{\alpha}}{B_{\alpha}}\right) \leq n^{1-(1+\delta )^2 +o(1)}
\end{align}
so it suffices to bound the second term on the right hand side of \eqref{light tail split}. 

Since $\tilde{X}^{(1)}\sim \mathcal{G}(n,q)$ for $q\leq \frac{2d'}{n(\log n)^{\varepsilon}}$, the series of events $\mathcal{D}$, $\mathcal{C}$, $\mathcal{E}$ defined in Section 3 applies to $\tilde{X}^{(1)}$.
Now define the event $\mathcal{K}_0$ measurable with respect to $\tilde{X}^{(1)}$:
\begin{align}
\mathcal{K}_0 := \mathcal{W} \cap \mathcal{D}_{(1+\delta )^2-1}\cap \mathcal{C}_{\varepsilon ,(1+\delta )^2-1}\cap \mathcal{E}_{(1+\delta )^2-1}\cap \mathcal{P}_{\kappa}\cap \mathcal{L}_{\delta ,\kappa}.
\end{align}
Using previous results,
\begin{align} \label{events}
& \mathbb{P} (\mathcal{W}^c) \leq e^{-\omega (\log n)} \quad \text{by Lemma \ref{star decomposition}},\nonumber\\
& \mathbb{P} (\mathcal{D}_{(1+\delta )^2-1}^c ) \leq n^{1-(1+\delta )^2+o(1)} \quad \text{by Lemma \ref{d}},\nonumber\\
& \mathbb{P} (\mathcal{C}_{\varepsilon ,(1+\delta )^2-1}^c ) \leq n^{1-(1+\delta )^2+o(1)} \quad \text{by Lemma \ref{c}},\nonumber\\
& \mathbb{P} (\mathcal{E}_{(1+\delta )^2-1}^c) \leq n^{1-(1+\delta )^2+o(1)} \quad \text{by Lemma \ref{e}},\\
& \mathbb{P}(\mathcal{P}_{\kappa}^c) \leq n^{-(1+\delta )^2+o(1)} \quad \text{by Lemma \ref{number of large degree}},\nonumber\\
& \mathbb{P}(\mathcal{L}_{\delta ,\kappa }^c) \leq n^{-(1+\delta )^2+o(1)} \quad \text{by \eqref{l bound}.}\nonumber
\end{align}
Combining the results above,
\begin{align} \label{event k0}
\mathbb{P}(\mathcal{K}_0^c) \leq n^{1-(1+\delta )^2+o(1)}.
\end{align}
By defining $\delta '$ as
\begin{align}
(1+\delta )\left( 1-\frac{\varepsilon^{\frac{1}{\alpha}}}{B_{\alpha}}\right) =(1+\delta ') +\varepsilon (1+\delta ),
\end{align}
we have
\begin{align} \label{z1 split}
\mathbb{P} \left( \| Z^{(1)}\| \geq (1+\delta )\left( 1-\frac{\varepsilon^{\frac{1}{\alpha}}}{B_{\alpha}}\right) \lambda_{\alpha}\right) \leq & \, \mathbb{E}\left[ \mathbb{P}\left( \| Z_1^{(1)}\| \geq (1+\delta ')\lambda_{\alpha} \, | \, \tilde{X}^{(1)}\right) \mathds{1}_{\mathcal{K}_0}\right] \nonumber\\
& +\mathbb{E}\left[ \mathbb{P} \left( \| Z_2^{(1)} \| \geq \varepsilon (1+\delta )\lambda_{\alpha} \, | \, \tilde{X}^{(1)}\right) \mathds{1}_{\mathcal{K}_0}\right] +\mathbb{P} (\mathcal{K}_0^c) .
\end{align}
From now we condition on the event $\mathcal{K}_0$. Since $\mathcal{K}_0$ implies $\mathcal{W}$, Lemma \ref{star decomposition} gives a decomposition of $\tilde{X}^{(1)}$ into $\tilde{X}_1^{(1)}$ and $\tilde{X}_2^{(1)}$ and the induced decomposition of $X^{(1)}$ into $X_1^{(1)}$ and $X_2^{(1)}$. Similarly, the corresponding matrix $Z^{(1)}$ is decomposed into $Z_1^{(1)}$ and $Z_2^{(1)}$.

\emph{Contribution from $Z_2^{(1)}$.} We check the conditions to apply Proposition \ref{infinity}. Conditioned on $\mathcal{K}_0$, if we let $C_1,\ldots ,C_m$ be the components of $X_2^{(1)}$, then the symmetrizations $\tilde{C}_1,\ldots ,\tilde{C}_m$ are the components of $\tilde{X}_2^{(1)}$. The definition of $X_2^{(1)}$ implies $d_1(X_2^{(1)}) =o\left(\frac{\log n}{\log \log n}\right)$. Since $\mathcal{K}_0$ implies $\mathcal{C}_{\varepsilon ,(1+\delta )^2-1}\cap \mathcal{E}_{(1+\delta )^2-1}$,
\begin{align*}
|V(C_k)| & = |V(\tilde{C}_k)| \leq \frac{(1+\delta )^2-1}{\varepsilon}\frac{\log n}{\log \log n}\\
|E(\tilde{C}_k)| & \leq |V(\tilde{C}_k)| +(1+\delta )^2-1 
\end{align*}
for each $1\leq k\leq m$.
Let $\mathcal{F}_{\delta_4}^k$ be the event that the number of self-loops in $C_k$ is less than $\delta_4$. For simplicity, let $a=\frac{(1+\delta )^2-1}{\varepsilon}$. By Lemma \ref{binomial loglog estimate},
\begin{align}
\mathbb{P}\left( (\mathcal{F}_{\delta_4}^k)^c \, | \, \tilde{X}^{(1)} \right) \mathds{1}_{\mathcal{K}_0} & \leq \mathbb{P}\left( \mathrm{Binom}\left( a\frac{\log n}{\log \log n} ,\frac{d}{n} \right) \geq \delta_4 \right) \leq n^{-\delta_4 +o(1)}.
\end{align}
Define $\mathcal{F}_{\delta_4}=\bigcap_{k=1}^{m}\mathcal{F}_{\delta_4}^k$. Then
\begin{align} \label{f estimate}
\mathbb{P}\left( (\mathcal{F}_{\delta_4})^c \, | \, \tilde{X}^{(1)}\right) \mathds{1}_{\mathcal{K}_0} \leq \sum_{k=1}^{m}\mathbb{P}\left( (\mathcal{F}_{\delta_4}^k)^c\, | \, \tilde{X}^{(1)}\right)\mathds{1}_{\mathcal{K}_0} \leq n^{1-\delta_4+o(1)}.
\end{align}
Define $Z_{2,k}^{(1)}$ be the component of $Z_2^{(1)}$ with underlying graph $C_k$.
On the event $\mathcal{K}_0\cap \mathcal{F}_{1+(1+\delta )^2}$, we may apply Proposition \ref{infinity} to the networks $Z_{2,1}^{(1)},\ldots ,Z_{2,m}^{(1)}$.
\begin{align} \label{39}
& \mathbb{P}\left( \| Z_2^{(1)}\|  \geq \varepsilon (1+\delta )\lambda_{\alpha}\, | \, \tilde{X}^{(1)}\right) \mathds{1}_{\mathcal{K}_0} \nonumber\\
& \leq \sup_{1\leq k\leq m}\mathbb{P}\left( \| Z_{2,k}^{(1)}\| \geq \varepsilon (1+\delta )\lambda_{\alpha}\, | \, \tilde{X}^{(1)}\right) \mathds{1}_{\mathcal{K}_0\cap \mathcal{F}_{1+(1+\delta )^2}}+\mathbb{P}\left( (\mathcal{F}_{1+(1+\delta )^2})^c\, | \, \tilde{X}^{(1)}\right)\mathds{1}_{\mathcal{K}_0} \nonumber\\
& \leq n^{-\omega (\log n)}+n^{-(1+\delta )^2+o(1)}=n^{-(1+\delta )^2+o(1)}.
\end{align}

\emph{Contribution from $Z_1^{(1)}$.}
We follow the proof of the contribution from $Z_1^{(1)}$ in \cite{ghn22}. Let $M$ be an integer satisfying
\begin{align*}
\frac{(1+\delta )^2}{\kappa} < M \leq \frac{(1+\delta )^2}{\kappa}+1.
\end{align*}
Conditioned on the event $\mathcal{D}_{(1+\delta )^2-1}$, $d_1(\tilde{X}_1^{(1)})\leq d_1(\tilde{X}^{(1)})\leq (1+\delta )^2\frac{\log n}{\log \log n}$. Thus for any star $S$ in $X_1^{(1)}$, the degree of $\tilde{S}$ in $\tilde{X}_1^{(1)}$ falls into at least one of the following categories:
\begin{enumerate}
\item $d(\tilde{S})\leq g(\kappa )$.
\item $d(\tilde{S})\in (g(i\kappa ),g((i+2)\kappa )]$ for $i=1,\ldots ,M-1$ and $i\neq m+1$.
\end{enumerate}

The contribution from the first category is bounded using Lemma \ref{single directed star} (with $\gamma =\kappa$ and $\rho =\delta '$), for sufficiently small $\kappa ,\varepsilon >0$,
\begin{align} \label{first category}
\liminf_{n\to\infty} & \frac{-\log \mathbb{E}\left[ \mathbb{P}\left( \max_{S\in \mathcal{S}, d(\tilde{S})\leq g(\kappa )} \| S\| \geq (1+\delta ')\lambda_{\alpha}\, | \, \tilde{X}^{(1)}\right) \mathds{1}_{\mathcal{K}_0}\right]}{\log n}\nonumber\\
& \geq -1+(1+\delta ')^{\alpha}\frac{2}{\alpha -2}\left( 1-\frac{2}{\alpha}\right)^{\frac{\alpha}{2}}\kappa^{1-\frac{\alpha}{2}}-\varepsilon \kappa >(1+\delta )^2 -1.
\end{align}
The `$-1$' term comes from the union bound in the first inequality and the second inequality holds for small $\kappa >0$ since $\alpha >2$.

For the second category, apply Lemma \ref{modified lemma 5.7} (with $\gamma =i\kappa ,h=2\kappa$ and $\rho =\delta '$) for each $i=1,\ldots ,M-1$ with $i\neq m+1$:
\begin{align} \label{second category}
& \liminf_{n\to\infty} \frac{-\log \mathbb{E}\left[ \mathbb{P}\left( \max_{S\in\mathcal{S}, d(\tilde{S})\in (g(i\kappa ),g((i+2)\kappa )]}\| S\| \geq (1+\delta ')\lambda_{\alpha}\, | \, \tilde{X}^{(1)}\right) \mathds{1}_{\mathcal{K}_0}\right]}{\log n}\nonumber\\
& \geq -f_{\alpha ,\delta '}((i+2)\kappa )-2\kappa -\kappa -\varepsilon (i+2)\kappa\nonumber\\
& \geq (1+\delta ')^2 -1-3\kappa -\varepsilon (M+1)\kappa \geq (1+\delta ')^2 -1-3\kappa -\varepsilon \left( \frac{(1+\delta )^2}{\kappa}+2\right)\kappa =: L.
\end{align}
Since all stars in $X_1^{(1)}$ are covered by these two categories, a union bound of \eqref{first category} and \eqref{second category} yields
\begin{align} \label{stars}
\liminf_{n\to\infty}\frac{-\log \mathbb{E}[\mathbb{P}(\| Z_1^{(1)}\|\geq (1+\delta ')\lambda_{\alpha} \, | \, \tilde{X}^{(1)})\mathds{1}_{\mathcal{K}_0}]}{\log n}\geq \min \lbrace (1+\delta )^2-1, L\rbrace .
\end{align}
Applying the bounds \eqref{event k0}, \eqref{39}, and \eqref{stars} to \eqref{z1 split} yields
\begin{align}
\liminf_{n\to\infty}\frac{-\log \mathbb{P}\left( \| Z^{(1)}\| \geq (1+\delta )\left( 1-\frac{\varepsilon^{\frac{1}{\alpha}}}{B_{\alpha}}\right) \lambda_{\alpha}\right)}{\log n} \geq \min \lbrace (1+\delta )^2-1, L\rbrace .
\end{align}
Combining this with \eqref{light tail split} and \eqref{37} we obtain
\begin{align}
\liminf_{n\to\infty}\frac{-\log \mathbb{P}(\| Z\| \geq (1+\delta )\lambda_{\alpha})}{\log n}\geq \min \lbrace (1+\delta )^2-1, L\rbrace .
\end{align}
Since $\lim_{\varepsilon \to 0}\delta '=\delta$ and $L$ is sufficiently close to $(1+\delta )^2-1$ for small $\varepsilon ,\kappa >0$, the proof is complete.

\subsection{Lower tail estimate}

Here we prove the following lower tail estimate:

\ltlt*

\subsubsection{Lower bound for the lower tail}

Define
\begin{align} \label{54}
\gamma_{\delta}' := (1-\delta )^2 \left( 1-\frac{2}{\alpha}\right) .
\end{align}
We start estimating the lower tail probability for a group of directed stars. The proof is almost identical to the corresponding case in \cite{ghn22}. Some slight changes will be conditioning on $\tilde{X}^{(1)}$ instead of $X^{(1)}$ for the application of structural lemmas, and using directed stars instead of undirected ones. 
The following is a directed version of \cite[Lemma 5.10]{ghn22}:

\begin{lemma}[{\cite[Lemma 5.10]{ghn22}}] \label{modified lemma 5.10}
Suppose that $\gamma >0$, $h\geq 0$ and $\kappa \geq \gamma -1$. Let $\mathcal{S}$ be the collection of at most $n^{1-\gamma -\kappa}$ vertex-disjoint weighted directed stars of size less than $g(\gamma +h)$. Assume that edge-weights are i.i.d. Weibull distributions with shape parameter $\alpha >2$ conditioned to be greater than $(\varepsilon \log \log n)^{\frac{1}{\alpha}}$ in absolute value. Then, for any $0<\rho <1$,
\begin{align}
\limsup_{n\to\infty}\frac{1}{\log n}\left( \log \log \frac{1}{\mathbb{P}(\max_{S\in \mathcal{S}}\| S\| \leq (1-\rho )\lambda_{\alpha})}\right) \leq f_{\alpha ,-\rho }(\gamma +h) +\kappa +h +\varepsilon (\gamma +h) .
\end{align}
\end{lemma}

\begin{proof}
As in the proof of Lemma \ref{single directed star}, choose $\lbrace \tilde{Y}_i\rbrace$ to be i.i.d. conditioned Weibull distributions.
By Lemma \ref{directed star} and the tail bound \eqref{152} with $d=1-\rho$ and $b=\gamma +h$,
\begin{align*}
\mathbb{P}(\| S\| \geq (1-\rho )\lambda_{\alpha}) & \leq 2\mathbb{P}\left( \tilde{Y}_1^2 +\cdots +\tilde{Y}_{g(\gamma +h)}^2 \geq (1-\rho )^2 \lambda_{\alpha}^2\right) .
\end{align*}
Given this, rest of the proof is identical to that of \cite[Lemma 5.10]{ghn22}.
\end{proof}

Now we begin the proof of the lower bound of Theorem \ref{ltlt}. Define the event $\mathcal{B}_{\delta}$ measurable with respect to $X$:
\begin{align}
\mathcal{B}_{\delta}:= \left\lbrace \| X\| \leq (1+\delta )\frac{(\log n)^{\frac{1}{2}}}{(\log \log n)^{\frac{1}{2}}}\right\rbrace .
\end{align}
Also define the event
\begin{align}
\mathcal{R}_{\kappa} := \lbrace \deg (\tilde{X}^{(1)}) < g(1+\kappa )\rbrace .
\end{align}
Now we define the following event measurable with respect to $\lbrace X, \tilde{X}^{(1)}\rbrace$:
\begin{align}
\mathcal{K}_1 := \mathcal{B}_{\delta} \cap \mathcal{W} \cap \mathcal{C}_{\varepsilon ,(1+\delta )^2-1}\cap \mathcal{E}_{(1+\delta )^2-1}\cap \mathcal{P}_{\kappa}\cap \mathcal{R}_{\kappa}.
\end{align}
Lemma \ref{graph eigenvalue} implies $\lim_{n\to\infty}\mathbb{P}(\mathcal{B}_{\delta})=1$. Since $\tilde{X}^{(1)}$ is stochastically dominated by $\mathcal{G}(n,\frac{d}{n})$, Lemma \ref{degrees} implies $\lim_{n\to\infty}\mathbb{P}(\mathcal{R}_{\kappa})=1$. Applying \eqref{events}, 
\begin{align} \label{59}
\lim_{n\to\infty}\mathbb{P}(\mathcal{K}_1)=1.
\end{align}
Since $\| Z\| \leq \| Z_1^{(1)}\| +\| Z_2^{(1)}\| +\| Z^{(2)}\|$, defining $\delta '$ by
\begin{align} \label{60}
1-\delta =(1-\delta ') +(1+\delta )\left( \varepsilon +\frac{\varepsilon^{\frac{1}{\alpha}}}{B_{\alpha}}\right) ,
\end{align}
we have
\begin{multline} \label{conditional k1}
\mathbb{P}(\| Z\| \leq (1-\delta )\lambda_{\alpha})  \geq \mathbb{E}\left[ \mathbb{P} \left( \| Z_1^{(1)}\| \leq (1-\delta ')\lambda_{\alpha}, \| Z_2^{(1)}\| \leq (1+\delta )\varepsilon \lambda_{\alpha}, \right. \right. \\
\left. \left. \| Z^{(2)}\| \leq (1+\delta )\frac{\varepsilon^{\frac{1}{\alpha}}}{B_{\alpha}}\lambda_{\alpha}\, \Big | \, X, \tilde{X}^{(1)}\right) \mathds{1}_{\mathcal{K}_1}\right] .
\end{multline}
Conditioned on the event $\mathcal{B}_{\delta}$, hence on $\mathcal{K}_1$,
\begin{align}
\| Z^{(2)}\| \leq (\varepsilon \log \log n)^{\frac{1}{\alpha}}\| X\| \leq (1+\delta )\frac{\varepsilon^{\frac{1}{\alpha}}}{B_{\alpha}}\lambda_{\alpha}.
\end{align}
Since $Z_1^{(1)}$, $Z_2^{(1)}$, and $X$ are conditionally independent given $\tilde{X}^{(1)}$, the conditional probability in the right hand side of \eqref{conditional k1} is
\begin{align}
& \mathbb{P}\left( \| Z_1^{(1)}\| \leq (1-\delta ')\lambda_{\alpha} , \| Z_2^{(1)}\| \leq (1+\delta )\varepsilon \lambda_{\alpha}\, | \, X, \tilde{X}^{(1)}\right)\nonumber\\
= & \, \mathbb{P}\left( \| Z_1^{(1)}\| \leq (1-\delta ')\lambda_{\alpha}, \| Z_2^{(1)}\| \leq (1+\delta )\varepsilon \lambda_{\alpha}\, | \, \tilde{X}^{(1)}\right)\nonumber\\
= & \, \mathbb{P} \left( \| Z_1^{(1)}\| \leq (1-\delta ')\lambda_{\alpha}\, | \, \tilde{X}^{(1)}\right) \mathbb{P} \left( \| Z_2^{(1)}\| \leq (1+\delta )\varepsilon \lambda_{\alpha}\, | \, \tilde{X}^{(1)}\right) .
\end{align}

\emph{Contribution from $Z_2^{(1)}$.} Since the event $\mathcal{K}_1$ implies $\mathcal{C}_{\varepsilon ,(1+\delta )^2-1}\cap \mathcal{E}_{(1+\delta )^2-1}$, using the high probability event $\mathcal{F}_{1+\delta}$, Proposition \ref{infinity} implies
\begin{align*}
\mathbb{P}\left( \| Z_2^{(1)}\| \geq (1+\delta )\varepsilon \lambda_{\alpha}\, | \, \tilde{X}^{(1)}\right) \leq n^{-\delta +o(1)}.
\end{align*}
This implies, for large $n$,
\begin{align} \label{65}
\mathbb{P}\left( \| Z_2^{(1)}\| \leq (1+\delta )\varepsilon \lambda_{\alpha}\, | \, \tilde{X}^{(1)}\right) \geq \frac{1}{2}.
\end{align}

\emph{Contribution from $Z_1^{(1)}$.} Define the events
\begin{align*}
\mathcal{J}_0 & := \left\lbrace \max_{\substack{S\in \mathcal{S} \\ d(S)\leq g(\kappa )}}\| S\| \leq (1-\delta ')\lambda_{\alpha}\right\rbrace ,\\
\mathcal{J}_i & := \left\lbrace \max_{\substack{S\in \mathcal{S} \\ d(S)\in (g(i\kappa ), g((i+1)\kappa )]}}\| S\| \leq (1-\delta ')\lambda_{\alpha}\right\rbrace \quad \text{for $1\leq i\leq m-1$},\\
\mathcal{J}_m & := \left\lbrace \max_{\substack{S\in \mathcal{S} \\ d(S) \in (g(m\kappa ),g((m+2)\kappa )]}}\| S\| \leq (1-\delta ')\lambda_{\alpha}\right\rbrace ,
\end{align*}
which controls the contribution from directed stars with similar degrees. Recall that $m$ is an integer satisfying $m\kappa <1\leq (m+1)\kappa$ so $(m+2)\kappa >1+\kappa$. Conditioned on $\mathcal{R}_{\kappa}$, and thus on $\mathcal{K}_1$, $\bigcap_{i=0}^{m}\mathcal{J}_i$ implies $\| Z_1^{(1)}\| \leq (1-\delta ')\lambda_{\alpha}$. Since a directed star $S\in \mathcal{S}$ satisfies $d_1(S)\leq d_1(\tilde{S})$, the number of directed stars are controlled by the event $\mathcal{P}_{\kappa}$ in the same way as in \cite{ghn22}. Thus the following lower bound estimates for $\mathbb{P}(\mathcal{J}_i \, | \, \tilde{X}^{(1)})$ are identical, and \cite[(104)]{ghn22} holds:
\begin{align}
\mathbb{P}\left( \| Z_1^{(1)}\| \leq (1-\delta ')\lambda_{\alpha}\, | \, \tilde{X}^{(1)}\right) \geq \exp \left( -\left( \frac{1}{\kappa}+2\right) n^{1-(1-\delta ')^2 +3\kappa +\varepsilon (1+2\kappa )+o(1)}\right) .
\end{align}
Applying this and \eqref{65} to \eqref{conditional k1},
\begin{align}
\mathbb{P}(\| Z\| \leq (1-\delta )\lambda_{\alpha}) \geq \frac{1}{2}\exp \left( n^{1-(1-\delta ')^2+3\kappa +\varepsilon (1+2\kappa )+o(1)}\right) \mathbb{P}(\mathcal{K}_1).
\end{align}
Using \eqref{59}, $\lim_{\varepsilon \to 0}\delta '=\delta$, and taking $\kappa ,\varepsilon >0$ sufficiently small, we obtain the desired bound.

\subsubsection{Upper bound for the lower tail}

Recall the event $\mathcal{A}_{\gamma ,\rho}$ defined in Section \ref{lower bound for the upper tail}. By \eqref{32},
\begin{align}
\mathbb{P}\left( \mathcal{A}_{\gamma_{\delta}',\rho}^c\right) \leq e^{-n^{1-\gamma_{\delta '}+o(1)}}.
\end{align}
On the event $\mathcal{A}_{\gamma_{\delta}',\rho}$, there exists $m=\lceil n^{1-\gamma_{\delta '}-\rho}\rceil$ columns $Z_{\cdot \sigma (1)}\ldots ,Z_{\cdot \sigma (m)}$ with at least $\lceil \gamma_{\delta}'\frac{\log n}{\log \log n}\rceil$ nonzero coordinates. Since
\begin{align}
\| Z\| \geq \max_{k=1,\ldots ,m}\| Z_{\cdot \sigma (k)}\| =\max_{k=1,\ldots ,m}\sum_{i=1}^{n}Z_{i\sigma (k)}^2 ,
\end{align}
we have
\begin{align} \label{69}
\mathbb{P}(\| Z\| \leq (1-\delta )\lambda_{\alpha} ) \leq \mathbb{E}\left[ \mathbb{P}\left( \max_{k=1,\ldots ,m}\sum_{i=1}^{n}Z_{i\sigma (k)}^2 \leq (1-\delta )^2 \lambda_{\alpha}^2 \, | \, X\right) \mathds{1}_{\mathcal{A}_{\gamma_{\delta}',\rho}}\right] +\mathbb{P}(\mathcal{A}_{\gamma_{\delta}',\rho }^c) .
\end{align}
Conditioned on the event $\mathcal{A}_{\gamma_{\delta}',\rho}$, using the tail estimates for Weibull distributions in \eqref{151}, estimating as in \cite[Section 5.2.2]{ghn22},
\begin{align}
\mathbb{P}\left( \max_{k=1,\ldots ,m}\sum_{i=1}^{n}Z_{i\sigma (k)}^2 \leq (1-\delta )^2 \lambda_{\alpha}^2 \, | \, X\right) & \leq \left( 1-n^{-(1-\delta )^{\alpha}\frac{2}{\alpha -2}(1-\frac{2}{\alpha})^{\frac{\alpha}{2}}(\gamma_{\delta}')^{1-\frac{\alpha}{2}}+o(1)}\right)^m\nonumber\\
& \leq \exp \left( -n^{1-\gamma_{\delta}'-\rho -(1-\delta )^{\alpha}\frac{2}{\alpha -2}(1-\frac{2}{\alpha})^{\frac{\alpha}{2}}(\gamma_{\delta}')^{1-\frac{\alpha}{2}}+o(1)}\right) \nonumber\\
& \leq \exp \left( -n^{1-(1-\delta )^2-\rho +o(1)}\right) .
\end{align}
Note that by definition \eqref{54}, $\gamma_{\delta}' <(1-\delta )^2$, so $\exp (-n^{1-(1-\delta )^2-\rho +o(1)})$ is dominant in \eqref{69}, and we obtain the desired result by taking $\rho \to 0$.

\section{Heavy-tailed weights}

Here we consider heavy-tailed weights with $0<\alpha \leq 2$.
We will simply write $\lambda_{\alpha} =\lambda_{\alpha}^{\mathrm{heavy}}=(\log n)^{\frac{1}{\alpha}}$ in this section.

\subsection{Upper tail estimate}

Recall the following upper tail estimate for heavy-tailed weights:

\htut*

\subsubsection{Lower bound for the upper tail}

Define the event that the number of edges in the digraph $X$ is not unusually small:
\begin{align} \label{m}
\mathcal{M} := \left\lbrace |E(X)| \geq \frac{dn}{2}\right\rbrace .
\end{align}
Then
\begin{align*}
\mathbb{P}(\| Z\| \geq (1+\delta )\lambda_{\alpha}) \geq \mathbb{E}\left[ \mathbb{P}(\| Z\| \geq (1+\delta )\lambda_{\alpha}\, | \, X)\mathds{1}_{\mathcal{M}}\right] .
\end{align*}
Since the number of missing edges $n^2-E(X)$ has distribution $\mathrm{Binom}\, (n^2, 1-\frac{d}{n})$, Lemma \ref{entropy} and \ref{entropy lower bound} yield the following estimate for some constant $c>0$:
\begin{align} \label{mc bound}
\mathbb{P}(\mathcal{M}^c) =\mathbb{P}\left( n^2-|E(X)| >n^2-\frac{dn}{2}\right) \leq e^{-n^2 I_{1-\frac{d}{n}}(1-\frac{d}{2n})}\leq e^{-cn} .
\end{align}
By conditioning on $\mathcal{M}$ and using the independence of edge weights,
\begin{align}
\mathbb{P}\left( \max_{(i,j)\in E(X)}|Z_{ij}|\geq (1+\delta )\lambda_{\alpha}\, | \, X\right) \geq 1-\left( 1-Cn^{-(1+\delta )^{\alpha}}\right)^{\frac{dn}{2}}\geq n^{1-(1+\delta )^{\alpha}+o(1)}.
\end{align}
Combining the bounds above, we have the desired estimate.

\subsubsection{Upper bound for the upper tail}

We apply the same decomposition of $Z$ into $Z^{(1)}+Z^{(2)}$ defined for the heavy-tailed case. Then we have the following analogue of Lemma \ref{light tail negligible}:

\begin{lemma} \label{heavy tail negligible}
For any $\delta >0$,
\begin{align*}
\liminf_{n\to\infty}\frac{-\log \mathbb{P}\left( \| Z^{(2)}\| \geq (1+\delta )(\varepsilon\log n)^{\frac{1}{\alpha}}\right)}{\log n}\geq (1+\delta )^2-1 .
\end{align*}
\end{lemma}

\begin{proof}
Note that $\| Z^{(2)}\| \leq (\varepsilon \log \log n)^{\frac{1}{\alpha}}\| X^{(2)}\| \leq (\varepsilon \log \log n)^{\frac{1}{\alpha}}\| X\|$ and  
\begin{align*}
\mathbb{P} \left( \| Z^{(2)}\| \geq (1+\delta )(\varepsilon \log n)^{\frac{1}{\alpha}}\right) \leq \mathbb{P}\left( \| X\| \geq \frac{(1+\delta )(\log n)^{\frac{1}{\alpha}}}{(\log \log n)^{\frac{1}{\alpha}}}\right) \leq \mathbb{P}\left( \| X\| \geq \frac{(1+\delta )(\log n)^{\frac{1}{2}}}{(\log \log n)^{\frac{1}{2}}}\right) .
\end{align*}
Applying Lemma \ref{graph ldp} finishes the proof.
\end{proof}

Here we present a directed version of \cite[Proposition 6.2]{ghn22} where the maximum clique size $k=2$. The original proposition involves a variational function $\phi_{\theta}$ defined as
\begin{align} \label{phi}
\phi_{\theta}(k) =\sup_{v=(v_1,\cdots ,v_k),\| v\|_1=1}\sum_{i,j\in [k], i\neq j}|v_i|^{\theta}|v_j|^{\theta}.
\end{align}
Since we are interested only in the case when $k=2$, it suffices to know the following property from \cite[(24)]{ghn22}:
\begin{align} \label{phi2}
\phi_{\theta}(2)=\frac{1}{2^{2\theta -1}}.
\end{align}

\begin{remark}
\upshape{
In the proof of \cite[Proposition 6.2]{ghn22}, the condition $\alpha <2$ was used only to apply the second inequality in \cite[Lemma 3.5]{ghn22} with $\theta =\frac{\beta}{2}>1$ where $\beta$ is the H\"older conjugate of $\alpha$. However, as stated in the lemma, the inequality actually holds for $\theta \geq 1$, so the same proof applies to the boundary (Gaussian) case $\alpha =2$. Therefore we shall apply \cite[Proposition 6.2]{ghn22} for $0<\alpha \leq 2$ in proving the proposition below.
}
\end{remark}

\begin{proposition} \label{heavy tail proposition}
Consider a connected directed network $G$ that satisfies
\begin{enumerate}
\item $d_1(G) \leq c_1\frac{\log n}{\log \log n}$,
\item $|V(G)| \leq \frac{c_2}{\varepsilon}\frac{\log n}{\log \log n}$,
\item $|E(\tilde{G})| \leq |V(\tilde{G})| +c_3$,
\item $|S(G)|\leq c_4$,
\end{enumerate}
where $\tilde{G}$ is the symmetrization of $G$ and $S(G)$ is the number of self-loops in $G$.
Suppose that the edge-weights are given by i.i.d. Weibull distributions with shape parameter $0<\alpha \leq 2$ conditioned to be greater than $(\varepsilon \log \log n)^{\frac{1}{\alpha}}$ in absolute value. Let $\rho >0$ and $0<\xi <(2c_3+c_4+2)^{\frac{1}{\beta}}\wedge \frac{1}{2}$ be constants, where $\beta$ is the H\"older conjugate of $\alpha$. For $\tau := (2c_3+c_4+2)^{\frac{1}{2\beta}}\xi^{\frac{1}{2}}$,
\begin{align} \label{66}
\mathbb{P}\left( \| G\| \geq \rho^{\frac{1}{\alpha}}\lambda_{\alpha}\right) \leq n^{-2^{-\alpha}\tau^{-\alpha}\rho +2c_2+o(1)}+n^{-\rho (1-\tau )^{\alpha}+c_1\xi^{-2}\varepsilon +o(1)}.
\end{align}
\end{proposition}

\begin{proof}
We proceed as in the proof of Proposition \ref{infinity} and consider the absolute values of edge-weights. Apply Proposition 3.5 to $G$ to get an undirected network $H^G$ with components $H_1^G, \ldots ,H_{\eta}^G$. Then each $H_i^G$ ($1\leq i\leq \eta$) satisfy the following structural properties:
\begin{gather*}
 d_1(H_i^G) \leq d_1(H^G)\leq d_1(G) \leq c_1\frac{\log n}{\log \log n},\\
 |V(H_i^G)| \leq |V(H^G)| \leq 2|V(G)| \leq \frac{2c_2}{\varepsilon}\frac{\log n}{\log \log n},\\
 |E(H_i^G)|-|V(H_i^G)| \leq 2(|E(\tilde{G})|-|V(\tilde{G})|) +S(G) \leq 2c_3+c_4.
\end{gather*}
Now $H_i^G$ is an undirected network with maximum clique size $2$, so we apply \cite[Proposition 6.2]{ghn22} with $k=2$. For $1<\alpha \leq 2$, using \eqref{phi2},
\begin{align*}
\mathbb{P}\left( \lambda_1 (H_i^G)\geq \rho^{\frac{1}{\alpha}}\lambda_{\alpha}\right) & \leq n^{-2^{-\alpha}\tau^{-\alpha}\rho +2c_2+o(1)}+n^{-2^{-1}\phi_{\beta /2}(2)^{1-\alpha}(1-\tau )^{\alpha}\rho +c_1\xi^{-2}\varepsilon +o(1)}\\
& = n^{-2^{-\alpha}\tau^{-\alpha}\rho +2c_2+o(1)}+n^{-\rho (1-\tau )^{\alpha} +c_1\xi^{-2}\varepsilon +o(1)}.
\end{align*}
For $0<\alpha \leq 1$,
\begin{align*}
\mathbb{P}\left( \lambda_1 (H_i^G)\geq \rho^{\frac{1}{\alpha}}\lambda_{\alpha}\right) \leq n^{-2^{-\alpha }\xi^{-\alpha}\rho +2c_2+o(1)}+n^{-\rho (1-\xi )^{\alpha}+c_1\xi^{-2}\varepsilon +o(1)}.
\end{align*}
By our assumption that $\xi <(2c_3+c_4+2)^{\frac{1}{\beta}}$, $\tau >\xi$ and for $0<\alpha \leq 2$,
\begin{align} \label{tau xi}
\mathbb{P}\left( \lambda_1 (H_i^G)\geq \rho^{\frac{1}{\alpha}}\lambda_{\alpha}\right) & \leq n^{-2^{-\alpha}(\tau \vee \xi )^{-\alpha}\rho +2c_2+o(1)}+n^{-\rho (1-\tau \vee \xi )^{\alpha}+c_1\xi^{-2}\varepsilon +o(1)}\nonumber\\
& = n^{-2^{-\alpha}\tau^{-\alpha}\rho +2c_2+o(1)}+n^{-\rho (1-\tau )^{\alpha}+c_1\xi^{-2}\varepsilon +o(1)}.
\end{align}
Since $H_i^G$ is an undirected network with nonnegative edge-weights, its largest eigenvalue equals the spectral norm and
\begin{align} \label{hg eta}
\| H^G\| =\max_{1\leq i\leq \eta}\| H_i^G\| =\max_{1\leq i\leq \eta}\lambda_1 (H_i^G).
\end{align}
Since $\eta =O\left( \frac{\log n}{\log \log n}\right) =n^{o(1)}$, a union bound with \eqref{tau xi} and \eqref{hg eta} finishes the proof.
\end{proof}

\noindent \emph{Proof of upper bound of the upper tail.}
Since $\| Z\| \leq \| Z^{(1)}\| +\| Z^{(2)}\|$,
\begin{align} \label{heavy tail split}
\mathbb{P}(\| Z\| \geq (1+\delta )\lambda_{\alpha}) \leq \mathbb{P}\left( \| Z^{(1)}\| \geq (1+\delta ')\lambda_{\alpha}\right) +\mathbb{P}\left( \| Z^{(2)}\| \geq \varepsilon^{\frac{1}{\alpha}}(1+\delta )\lambda_{\alpha}\right) 
\end{align}
where $\delta '$ is defined as
\begin{align}
1+\delta =(1+\delta ')+\varepsilon^{\frac{1}{\alpha}}(1+\delta ) .
\end{align}
By Lemma \ref{heavy tail negligible},
\begin{align} \label{heavy tail z2 estimate}
\mathbb{P}\left( \| Z^{(2)}\| \geq \varepsilon^{\frac{1}{\alpha}}(1+\delta )\lambda_{\alpha}\right) \leq n^{1-(1+\delta )^2 +o(1)},
\end{align}
so $Z^{(2)}$ is spectrally negligible.

Let $\delta_0 =(1+\delta )^{\alpha}-1$ and define the following event:
\begin{align*}
\mathcal{K}_2 := \mathcal{D}_{\delta_0 }\cap \mathcal{C}_{\varepsilon ,\delta_0}\cap \mathcal{E}_{\delta_0}.
\end{align*}
By Lemmas \ref{d}, \ref{c}, \ref{e}, 
\begin{align} \label{21}
\mathbb{P} (\mathcal{K}_2^c) \leq n^{-\delta_0 +o(1)}.
\end{align} 
Conditioned on $\mathcal{K}_2$, let $C_1,\ldots ,C_m$ be the components of $X^{(1)}$ and $Z_1^{(1)},\ldots ,Z_m^{(1)}$ be the corresponding subnetworks of $Z^{(1)}$. As in the light-tailed case, let $\mathcal{F}_{\delta_4}^k$ be the event that the number of self-loops in $C_k$ is less than $\delta_4$, and $\mathcal{F}_{\delta_4}=\bigcap_{k=1}^{m}\mathcal{F}_{\delta_4}^k$. Applying \eqref{f estimate} with $\delta_4=1+\delta_0$,
\begin{align} \label{heavy tail f}
\mathbb{P}\left( (\mathcal{F}_{1+\delta_0})^c\, | \, \tilde{X}^{(1)}\right) \mathds{1}_{\mathcal{K}_2} \leq n^{-\delta_0+o(1)}.
\end{align}
For each $k=1,\ldots ,m$, we apply Proposition \ref{heavy tail proposition} to each $Z_{k}^{(1)}$ with
\begin{align}
c_1=c_2=1+\delta_0, \quad c_3=\delta_0, \quad \xi =\varepsilon^{\frac{1}{4}}, \quad \tau =(3\delta_0+3)^{\frac{1}{2\beta}}\varepsilon^{\frac{1}{8}}, \quad \rho =(1+\delta ')^{\alpha}.
\end{align}
Then the first term in \eqref{66} is negligible compared to the second term for small enough $\varepsilon >0$, and
\begin{align}
\mathbb{P}\left( \| Z_k^{(1)}\| \geq (1+\delta ')\lambda_{\alpha}\, | \, \tilde{X}^{(1)}\right) \mathds{1}_{\mathcal{K}_2\cap \mathcal{F}_{1+\delta_0}} \leq n^{-(1+\delta ')^{\alpha}(1-\tau )^{\alpha}+(1+\delta_0)\varepsilon^{1/2}+o(1)}.
\end{align}
Using \eqref{21}, \eqref{heavy tail f} with a union bound,
\begin{align} \label{77}
\mathbb{P}\left( \| Z^{(1)}\| \geq (1+\delta ')\lambda_{\alpha}\right) \leq n^{-\delta_0+o(1)}+n^{1-(1+\delta ')^{\alpha}(1-\tau )^{\alpha}+(1+\delta_0)\varepsilon^{1/2}+o(1)} \leq n^{-\delta_0+\zeta +o(1)}
\end{align}
for some $\zeta =\zeta (\varepsilon )>0$ where $\lim_{\varepsilon \to 0}\zeta =0$. Combining \eqref{heavy tail split}, \eqref{heavy tail z2 estimate}, \eqref{77} and taking $\varepsilon \to 0$ completes the proof.

\subsection{Lower tail estimates}

We finally prove the following lower tail estimate:

\htlt*

\subsubsection{Lower bound for the lower tail}

Define the $X$-measurable event
\begin{align}
\mathcal{B}_1 := \left\lbrace \| X\| \leq 2\frac{(\log n)^{\frac{1}{2}}}{(\log \log n)^{\frac{1}{2}}}\right\rbrace 
\end{align}
and the event $\mathcal{K}_3$ measurable with respect to $\lbrace X, \tilde{X}^{(1)}\rbrace$:
\begin{align}
\mathcal{K}_3 := \mathcal{D}_{\delta}\cap \mathcal{C}_{\varepsilon ,\delta}\cap \mathcal{E}_{\delta}\cap \mathcal{B}_1.
\end{align}
Lemma \ref{graph eigenvalue} implies $\lim_{n\to\infty}\mathbb{P}(\mathcal{B}_1)=1$, and combining this with Lemmas \ref{d}, \ref{c}, and \ref{e}, we have for large $n$,
\begin{align} \label{80}
\mathbb{P}(\mathcal{K}_3) \geq \frac{1}{2}.
\end{align}
Since $\| Z\| \leq \| Z^{(1)}\| +\| Z^{(2)}\|$, setting
\begin{align}
\delta '' := \delta +\varepsilon^{\frac{1}{\alpha}},
\end{align}
we have
\begin{align} \label{heavytail33}
\mathbb{P}(\| Z\| \leq (1-\delta )\lambda_{\alpha}) \geq \mathbb{E}\left[ \mathbb{P}\left( \| Z^{(1)}\| \leq (1-\delta '')\lambda_{\alpha}, \| Z^{(2)}\| \leq \varepsilon^{\frac{1}{\alpha}}\lambda_{\alpha}\, | \, X,\tilde{X}^{(1)}\right) \mathds{1}_{\mathcal{K}_3}\right] .
\end{align}
Under the event $\mathcal{B}_1$, hence under the event $\mathcal{K}_3$, for large enough $n$,
\begin{align}
\| Z^{(2)}\| \leq 2\frac{(\log n)^{\frac{1}{2}}}{(\log \log n)^{\frac{1}{2}}}\cdot (\varepsilon \log \log n)^{\frac{1}{\alpha}}\leq \varepsilon^{\frac{1}{\alpha}}\lambda_{\alpha}.
\end{align}
Since $X$ and $Z^{(1)}$ are conditionally independent given $\tilde{X}^{(1)}$, under the event $\mathcal{K}_3$,
\begin{align} \label{heavytail35}
\mathbb{P} & \left( \| Z^{(1)}\| \leq (1-\delta '') \lambda_{\alpha},\| Z^{(2)}\| \leq \varepsilon^{\frac{1}{\alpha}}\lambda_{\alpha}\, | \, X,\tilde{X}^{(1)}\right)\nonumber\\
& =\mathbb{P}\left( \| Z^{(1)}\| \leq (1-\delta '')\lambda_{\alpha}\, | \, X, \tilde{X}^{(1)}\right) =\mathbb{P}\left( \| Z^{(1)}\| \leq (1-\delta '')\lambda_{\alpha}\, | \, \tilde{X}^{(1)}\right) .
\end{align}
Combining \eqref{heavytail33} and \eqref{heavytail35},
\begin{align} \label{heavytail36}
\mathbb{P}(\| Z\| \leq (1-\delta )\lambda_{\alpha}) \geq \mathbb{E}\left[ \mathbb{P}\left( \| Z^{(1)}\| \leq (1-\delta '')\lambda_{\alpha}\, | \, \tilde{X}^{(1)}\right) \mathds{1}_{\mathcal{K}_3}\right] .
\end{align}
Let $C_1,\ldots ,C_m$ be the components of $X^{(1)}$ and $Z_1^{(1)},\ldots ,Z_m^{(1)}$ be the corresponding subnetworks of $Z^{(1)}$. Recall the events $\mathcal{F}_{\delta_4}$ defined in the proof of the upper tail. Under the event $\mathcal{K}_3\cap \mathcal{F}_{1+(1-\delta )^{\alpha}}$, we apply Proposition \ref{heavy tail proposition} to $Z_k^{(1)}$ for each $k=1,\ldots ,m$ with
\begin{align}
c_1=c_2=1+\delta ,\quad c_3=\delta ,\quad \xi =\varepsilon^{\frac{1}{4}},\quad \tau =( 2\delta +(1-\delta )^{\alpha}+3)^{\frac{1}{2\beta}}\varepsilon^{\frac{1}{8}}, \quad \rho =(1-\delta '')^{\alpha}.
\end{align}
Then the dominating term in \eqref{66} is the second term for sufficiently small $\varepsilon >0$, and
\begin{align*}
\mathbb{P}\left( \| Z_k^{(1)}\| \geq (1-\delta '' )\lambda_{\alpha}\, | \, \tilde{X}^{(1)}\right) \mathds{1}_{\mathcal{K}_3\cap \mathcal{F}_{1+(1-\delta )^{\alpha}}} & \leq   n^{-(1-\delta '')^{\alpha}(1-\tau )^{\alpha}+(1+\delta )\varepsilon^{1/2}+o(1)} \leq n^{-(1-\delta )^{\alpha}+\zeta +o(1)}
\end{align*}
for some $\zeta =\zeta (\varepsilon ) >0$ with $\lim_{\varepsilon \to 0}\zeta =0$. Using \eqref{f estimate},
\begin{align*}
\mathbb{P}\left( \| Z_k^{(1)}\| \geq (1-\delta '')\lambda_{\alpha}\, | \, \tilde{X}^{(1)}\right) \mathds{1}_{\mathcal{K}_3}\leq n^{-(1-\delta )^{\alpha}+o(1)}+n^{-(1-\delta )^{\alpha}+\zeta +o(1)}=n^{-(1-\delta )^{\alpha}+\zeta +o(1)}.
\end{align*}
Thus under the event $\mathcal{K}_3$,
\begin{align} \label{87}
\mathbb{P}\left( \| Z^{(1)}\| \leq (1-\delta '' )\lambda_{\alpha}\, | \, \tilde{X}^{(1)}\right) \geq \left( 1-n^{-(1-\delta )^{\alpha}+\zeta +o(1)}\right)^n \geq e^{-n^{1-(1-\delta )^{\alpha}+\zeta +o(1)}}.
\end{align}
Applying this and \eqref{80} to \eqref{heavytail36},
\begin{align}
\mathbb{P}(\| Z\| \leq (1-\delta )\lambda_{\alpha})\geq \frac{1}{2}e^{-n^{1-(1-\delta )^{\alpha}+\zeta +o(1)}}.
\end{align}
The proof follows by taking $\varepsilon \to 0$.

\subsubsection{Upper bound for the lower tail}

Recall the event $\mathcal{M} =\left\lbrace |E(X)| \geq \frac{dn}{2}\right\rbrace$ defined in \eqref{m}. 
\begin{align}
\mathbb{P}\left( \| Z\| \leq (1-\delta )\lambda_{\alpha}\right) \leq \mathbb{E}\left[ \mathbb{P}\left( \max_{(i,j)\in E(X)}|Y_{ij}|\leq (1-\delta )\lambda_{\alpha}\, | \, X\right) \mathds{1}_{\mathcal{M}}\right] +\mathbb{P}(\mathcal{M}^c) .
\end{align}
Using the tail behavior of Weibull distributions,
\begin{align}
\mathbb{P}\left( \max_{(i,j)\in E(X)}|Y_{ij}| \leq (1-\delta )\lambda_{\alpha}\, | \, X\right) \mathds{1}_{\mathcal{M}}\leq \left( 1-Cn^{-(1-\delta )^{\alpha}}\right)^{\frac{dn}{2}}\leq e^{-n^{1-(1-\delta )^{\alpha}+o(1)}}.
\end{align}
Combining this with the bound for $\mathbb{P}(\mathcal{M}^c)$ in \eqref{mc bound}, the proof is complete.

\section{Appendix}

Here we restate some useful results regarding the tail behavior of sums of Weibull distributions and estimates for the relative entropy functional from the appendix of \cite{ghn22}.

\subsection{Tail estimates for sums of Weibull distributions}

We consider tail bounds for the sum of squares of i.i.d. Weibull random variables $\lbrace Y_i\rbrace_{i=1,2,\cdots}$ and the conditioned variables $\tilde{Y}_i =Y_i \mathds{1}_{|Y_i| >(\varepsilon \log \log n)^{\frac{1}{\alpha}}}$.
Recall that for a Weibull random variable $W$ there exists constants $C_1, C_2>0$ such that
\begin{align*}
C_1 e^{-t^{\alpha}}\leq \mathbb{P}(| W|\geq t)\leq C_2 e^{-t^{\alpha}}
\end{align*}
for all $t>0$. As in the previous sections we denote the typical value for light-tailed distributions by
\begin{align*}
\lambda_{\alpha}^{\mathrm{light}} = \left( \frac{2}{\alpha}\right)^{\frac{1}{\alpha}} \left( 1-\frac{2}{\alpha}\right)^{\frac{1}{2}-\frac{1}{\alpha}}\frac{(\log n)^{\frac{1}{2}}}{(\log \log n)^{\frac{1}{2}-\frac{1}{\alpha}}} .
\end{align*}

\begin{lemma}[{\cite[Lemma 7.1]{ghn22}}] \label{weibull lt estimate}
Suppose that $\alpha >2$, $t=d^2 \left( \lambda_{\alpha}^{\mathrm{light}}\right)^2 +o \left( \frac{\log n}{\log \log n}\right)$ and $k=b\frac{\log n}{\log \log n}+o\left( \frac{\log n}{\log \log n}\right)$ for some constants $b,d>0$. Then,
\begin{align} \label{151}
\lim_{n\to\infty} -\frac{\log \mathbb{P}(Y_1^2 +\cdots +Y_k^2 \geq t)}{\log n} =d^{\alpha}\frac{2}{\alpha -2}\left( 1-\frac{2}{\alpha}\right)^{\frac{\alpha}{2}}b^{1-\frac{\alpha}{2}}
\end{align}
and
\begin{align} \label{152}
\lim_{n\to\infty} -\frac{\log \mathbb{P}\left( \tilde{Y}_1^2 +\cdots +\tilde{Y}_k^2 \geq t\right)}{\log n} =d^{\alpha}\frac{2}{\alpha -2}\left( 1-\frac{2}{\alpha}\right)^{\frac{\alpha}{2}}b^{1-\frac{\alpha}{2}}-b\varepsilon .
\end{align}
\end{lemma}

\subsection{Estimates for relative entropy}

We define the relative entropy functional $I_p$ by
\begin{align}
I_p(q) := q\log \frac{q}{p} +(1-q)\log \frac{1-q}{1-p}.
\end{align}
The following lower bound estimate holds for $I_p$ under certain conditions.

\begin{lemma}[{\cite[Lemma 7.6]{ghn22}}] \label{entropy lower bound}
There exists a constant $c>0$ such that for $0<p<1$ and $0<q\leq \frac{p}{2}$,
\begin{align*}
I_p(q) \geq cp.
\end{align*}
\end{lemma}

\begin{lemma}[{\cite[Lemma 3.3]{lz17}}] \label{entropy asymptotic}
If $0<p\ll q$ and $q\leq 1-p$ then
\begin{align*}
I_p(p+q) =(1+o(1))q\log \frac{q}{p}.
\end{align*}
\end{lemma}

The tail estimates for the Binomial distribution can be stated in terms of the entropy functional:
\begin{lemma}[{\cite[Lemma 4.7.2]{ash}}] \label{entropy}
For $m\in \mathbb{N}$ and $0<q<1$, let $X$ have distribution $\mathrm{Binom}(m,q)$.
\begin{enumerate}
\item If $q<\theta <1$,
\begin{align} \label{binom upper tail}
\frac{1}{\sqrt{8m\theta (1-\theta )}}e^{-mI_q(\theta )}\leq \mathbb{P}(X\geq \theta m) \leq e^{-mI_q(\theta )}.
\end{align}
\item If $0<\theta <q <1$,
\begin{align} \label{binom lower tail}
\frac{1}{\sqrt{8m\theta (1-\theta )}}e^{-mI_q(\theta )}\leq \mathbb{P}(X\leq \theta m) \leq e^{-mI_q(\theta )}.
\end{align}
\end{enumerate}
\end{lemma}

Finally we prove a special case of the binomial tail estimate that has been used multiple times throughout the paper.
\begin{lemma} \label{binomial loglog estimate}
Let $a,d,\delta >0$ be constants. Then the following estimate holds:
\begin{align}
\mathbb{P}\left( \mathrm{Binom}\left( a\frac{\log n}{\log \log n},\frac{d}{n}\right) \geq \delta \right) \leq n^{-\delta +o(1)}.
\end{align}
\end{lemma}

\begin{proof}
Apply \eqref{binom upper tail} with $X\sim \mathrm{Binom}(m,q)$ where $m=a\frac{\log n}{\log \log n}$, $q=\frac{d}{n}$ and $\theta =\frac{\delta}{a}\frac{\log \log n}{\log n}$. Then
\begin{align*}
\mathbb{P}( X\geq \delta ) \leq e^{-\frac{a\log n}{\log \log n}I_{\frac{d}{n}}\left(\frac{\delta \log \log n}{a\log n}\right)} =e^{ -\frac{a\log n}{\log \log n}(1+o(1))\frac{\delta \log \log n}{a\log n}\log \frac{\delta n\log \log n}{ad\log n}}=e^{-(1+o(1))\delta \log n}=n^{-\delta +o(1)},
\end{align*}
where we used Lemma \ref{entropy asymptotic} in the first equality.
\end{proof}

\end{document}